\documentclass[aap]{imsart}

\RequirePackage{amsthm,amsmath,amsfonts,amssymb}
\RequirePackage[numbers]{natbib}
\RequirePackage[colorlinks,citecolor=blue,urlcolor=blue]{hyperref}
\RequirePackage{graphicx}

\startlocaldefs

\numberwithin{equation}{section}
\theoremstyle{plain}
\newtheorem{theorem}[]{Theorem}
\newtheorem*{theorem*}{Theorem}
\newtheorem{proposition}[equation]{Proposition}
\newtheorem{corollary}[equation]{Corollary}
\newtheorem*{corollary*}{Corollary}
\newtheorem{lemma}[equation]{Lemma}
\newtheorem*{lemma*}{Lemma}

\theoremstyle{definition}
\newtheorem{definition}[equation]{Definition}
\theoremstyle{remark}
\newtheorem{remark}[equation]{Remark}
\newtheorem*{remark*}{Remark}

\usepackage{amssymb}
\usepackage{dsfont}
\usepackage{subfig}
\usepackage[shortlabels]{enumitem}
\usepackage{mathrsfs}

\newcommand*{\E}{\mathbb{E}}
\newcommand*{\N}{\mathbb{N}}
\newcommand*{\Pb}{\mathbb{P}}

\newcommand*{\R}{\mathbb{R}}
\newcommand*{\Z}{\mathbb{Z}}
\newcommand{\1}{\mathds{1}}
\newcommand{\cL}{\mathcal{L}}

\newcommand{\cO}{\mathcal{O}}
\newcommand{\cR}{\mathcal{R}}
\newcommand{\cS}{\mathcal{S}}
\newcommand{\sE}{\mathscr{E}}
\newcommand{\dd}{\mathrm{d}}
\renewcommand{\geq}{\geqslant}
\renewcommand{\leq}{\leqslant}

\graphicspath{{figures/}}

\endlocaldefs

\begin{document}

\begin{frontmatter}
\title{On the One-Dimensional Contact \\ Process with Enhancements}
\runtitle{On the One-Dimensional Contact Process with Enhancements}

\begin{aug}
\author[A]{\fnms{Enrique}~\snm{Andjel}%
}
\and
\author[B]{\fnms{Leonardo~T.}~\snm{Rolla}%
}
\address[A]{Aix-Marseille University}
\address[B]{University of São Paulo}
\end{aug}

\begin{abstract}
We study a one-dimensional contact process with two infection parameters, one giving the infection rates at the boundaries of a finite infected region and the other one the rates within that region. We prove that the critical value of each of these parameters is a strictly monotone continuous function of the other parameter. We also show that if one of these parameters is equal to the critical value of the standard contact process and the other parameter is strictly larger, then the infection starting from a single point has positive probability of surviving. This is in contrast with another result also obtained here, that the critical contact process on the half line with enhanced infection rate at finitely many sites also dies out.
\end{abstract}

\begin{keyword}[class=MSC]
\kwd[Primary ]{60K35}
\kwd{82B43}
\kwd[; secondary ]{60D05}
\end{keyword}

\begin{keyword}
\kwd{Contact process}
\kwd{Percolation}
\end{keyword}

\end{frontmatter}

\section{Introduction}

In this paper we consider a one-dimensional contact process with modified boundaries. This model was introduced by Durrett and Schinazi in~\cite{DurrettSchinazi00}. It differs from the standard one-dimensional contact process only in the way the rightmost (leftmost) infected site transmits its infection to its right (left) nearest neighbor. These infection rates are given by a parameter $\lambda_e$ while infection everywhere else occurs at rate $\lambda_i$, where $e$ and $i$ refer to external and internal (as usual, we assume that the recovery rate is $1$).
When $ \lambda_e=\lambda_i=\lambda $, this process reduces to the case of standard contact process, and we refer the reader to~\cite{Griffeath81,Liggett05,Valesin24} for its basic properties.
(At the end of this introduction, we also consider another variant of the process.)

Let $ \theta(\lambda_i,\lambda_e) $ denote the probability that, starting with only the origin infected, the infection survives throughout time.

Although comparisons between different pairs $ (\lambda_i,\lambda_e) $ on the parameter space are not immediate, they can be compared with points $ (\lambda_i,\lambda_i) $ and $ (\lambda_e,\lambda_e) $ through any standard graphical construction (we describe one in the next section).
Using this fact, that the critical contact process dies out~\cite{BezuidenhoutGrimmett90}, and comparing the boundary of the process with a simple random walk, one immediately gets the picture of the phase space shown in Figure~\ref{fig:phasespace2a}.

It turns out that $ \theta(\lambda_i,\lambda_e) $ 
is non-decreasing in both parameters. We note that although this property is very intuitive, its proof is not immediate~\cite[Proposition~1]{DurrettSchinazi00}.
So we define
\begin{align*}
\lambda_c
&:= \inf\{\lambda : \theta(\lambda,\lambda)>0 \}
= \sup\{\lambda : \theta(\lambda,\lambda)=0 \}
,
\\
\lambda_*^i(\lambda_e)
&:= \inf\{\lambda : \theta(\lambda,\lambda_e)>0 \}
=
\sup\{\lambda : \theta(\lambda,\lambda_e)=0 \}
,
\\
\lambda_*^e(\lambda_i)
&:= \inf\{\lambda : \theta(\lambda_i,\lambda)>0 \}
=
\sup\{\lambda : \theta(\lambda_i,\lambda)=0 \}
.
\end{align*}

The research program reported here started with the question of whether $ \lambda_*^i(\lambda_e) $ would be strictly larger than $ \lambda_c $ for every $ \lambda_e<\lambda_c $.
This question resembles a typical question of enhancement: does a decrease in $ \lambda_e $, however small, have a strong enough effect on the dynamics so as to cause the critical value of $ \lambda_i $ to increase?
For percolation systems with two parameters, there are well established techniques that can often be used to prove that the critical value for one of the parameters is a strictly monotone function of the other parameter (see~\cite{AizenmanGrimmett91} and~\cite[Section~3.3]{Grimmett99}).
However, these techniques are much better adapted to the non-oriented case, and will usually break down for the contact process and other oriented models.
For instance, in the classic contact process, we are not aware of a proof that the critical parameter is strictly decreasing as we increase the spatial dimension (although $\lambda_c \to 0$ as $d\to\infty$~\cite{Griffeath83,HolleyLiggett81}).
Enhancement arguments being of no help, we have to rely on other methods which depend on the one-dimensional properties of our processes.

\begin{figure}[b!]
\hfil
\subfloat[]{\includegraphics[page=1,width=0.33\textwidth]{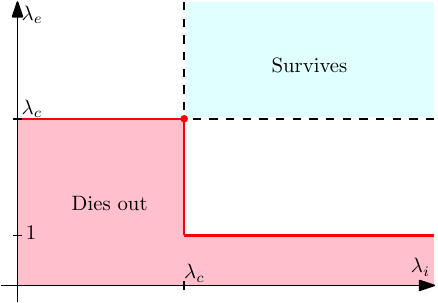}\label{fig:phasespace2a}}%
\hfil
\subfloat[]{\includegraphics[page=2,width=0.33\textwidth]{phasespace}\label{fig:phasespace2b}}%
\hfil
\subfloat[]{\includegraphics[page=3,width=0.33\textwidth]{phasespace}\label{fig:phasespace2c}}%
\hfil
\caption{(a) Simple properties of the phase space: $ \theta(\lambda_i,\lambda_e)>0 $ if $ \min\{\lambda_i,\lambda_e\}>\lambda_c $ and $ \theta(\lambda_i,\lambda_e)=0 $ if $ \max\{\lambda_i,\lambda_e\} \leq \lambda_c $ or $ \lambda_e \leq 1 $.
(b) Results of~\cite{DurrettSchinazi00}. (c) Results of this paper.
\\ \hspace*{\fill} (color~online)}
\end{figure}

The following were proved in~\cite{DurrettSchinazi00} and are illustrated in Figure~\ref{fig:phasespace2b}.
\begin{theorem}
For $\lambda_e > \lambda_c$, we have $\lambda_*^i(\lambda_e) = \lambda_c$.
\end{theorem}

\begin{theorem}
\label{thm:ds}
For $\lambda_e > 1$, we have $\lambda_*^i(\lambda_e) < \infty$.
\end{theorem}

\begin{corollary}
$\lim_{\lambda_i \rightarrow \infty} \lambda^e_*(\lambda_i)=1$.
\end{corollary}

We now state the main results of this paper, illustrated in Figure~\ref{fig:phasespace2c}.

The first and second theorems below state that the critical value of $\lambda_i$ is a strictly decreasing function of $\lambda_e$ in the relevant region, and vice versa.
As a consequence, the critical value of 
either parameter is a continuous function of the other parameter.

\begin{theorem}
\label{thm:noverticals}
The function
$\lambda_*^i(\lambda_e)$ is strictly decreasing for $\lambda_e \in (1,\lambda_c]$.
\end{theorem}

In particular, we answer the original question affirmatively:
\begin{corollary}
For $ \lambda_e < \lambda_c $, we have $\lambda_*^i(\lambda_e) > \lambda_c$.
\end{corollary}

\begin{theorem}
\label{thm:nohorizontals}
The function
$\lambda_*^e(\lambda_i)$ is strictly decreasing on $[\lambda_c,+\infty)$.
\end{theorem}

\begin{corollary}
For all $\lambda_i < \infty$ we have $\lambda_*^e(\lambda_i) > 1$.
\end{corollary}

\begin{corollary}
\label{cor:critint}
For $ \lambda_e = \lambda_c $, we have $\lambda_*^i(\lambda_e) = \lambda_c$.
\end{corollary}

The following corollary is a consequence of Theorems~\ref{thm:noverticals} and~\ref{thm:nohorizontals}.
For the first two parts, note that if one of the functions $\lambda^e_*$ and $\lambda^i_*$
is strictly monotone the other one must be continuous. Then these functions are inverses of each other and the last two parts of the corollary follow.
\begin{corollary}
\label{continuity1}
The critical curve has the following properties.
\begin{enumerate}[$ (i) $]
\item
The function $\lambda^e_*:[\lambda_c,\infty) \rightarrow (1,\lambda_c]$ is continuous.
\item
The function $\lambda^i_*:(1,\lambda_c] \rightarrow [\lambda_c,\infty)$ is continuous.
\item
$\lim_{\lambda_e \to 1} \lambda^i_*(\lambda_e)=+\infty$.
\item
For every $a \in [\lambda_c,\infty) $ and $b \in (1,\lambda_c]$, $\lambda^e_*(a)=b$ if and only if $\lambda^i_*(b)=a$.
\end{enumerate}
\end{corollary}

\begin{remark}
In the course of proving Theorems~\ref{thm:noverticals} and~\ref{thm:nohorizontals}, we show existence and uniqueness
of a stationary state for the process seen from the rightmost infected site, as well as convergence to such state, for supercritical pairs of parameters
(Proposition~\ref{prop:existsinv}).
\end{remark}

Corollary~\ref{cor:critint} is somewhat expected, it says that a process with supercritical parameter in the bulk and critical parameter at the boundary will survive with positive probability.
Perhaps more surprising is its counterpart that we state now.

\begin{theorem}
\label{thm:helpout}
For $ \lambda_i = \lambda_c $, we have $\lambda_*^e(\lambda_i) = \lambda_c$.
\end{theorem}

We make a few remarks about the above theorem.

\begin{remark}
\label{rmk:abitmore}
We will in fact prove something slightly stronger: if the underlying graph is $\Z_+$ instead of $\Z$, and the infection rate is $ \lambda_c + \varepsilon $ at the right-hand side boundary and $ \lambda_c $ everywhere else, then the process survives.
\end{remark}

\begin{remark}
Even though the theorem says that the process survives for $ \lambda_i=\lambda_c $ and $ \lambda_e>\lambda_c $, the process does not have a non-trivial invariant measure.
Indeed, if it had such a measure, it would be supported on configurations that have infinitely many occupied sites in both directions and would thus be invariant for the standard contact process with critical parameter, contradicting a.s.\ extinction established in~\cite{BezuidenhoutGrimmett90}.
\end{remark}

\begin{remark}
The critical contact process seen from the rightmost infected site has a unique stationary state~\cite{CoxDurrettSchinazi91}.
In the course of proving Theorem~\ref{thm:helpout}, we prove that the critical contact process started from this distribution has zero asymptotic speed (Proposition~\ref{prop:critinvspeedzero}).
\end{remark}

\begin{remark}
\label{rmk:zhang1}
It is conjectured that the critical contact process dies out even if we change the infection rate to an arbitrarily large value at a finite number of edges.
This conjecture is supported by an analogous result in non-oriented percolation~\cite{Zhang94}.
In contrast, the above theorem says that, if at each time we increase the infection rate by $ \varepsilon>0 $ at a specific, dynamically chosen, pair of edges (those at the boundary of the infected interval), then the process survives with positive probability.
\end{remark}

Regarding the conjecture mentioned in Remark~\ref{rmk:zhang1}, we obtain a partial result to be contrasted with Theorem~\ref{thm:helpout} and Remark~\ref{rmk:abitmore}.

\begin{theorem}
\label{thm:fixenhancement}
Consider the standard contact process on $\Z_+$, with recovery rate $1$ and infection rate $\lambda_c$ at all but finitely many sites, and having at finitely many sites a recovery rate $\delta>0$ and an infection rate $\lambda<\infty$ arbitrary but fixed. This process dies out a.s.
\end{theorem}

\begin{remark}
\label{rmk:zhang2}
Our proof relies strongly on the underlying graph being $\Z_+$ rather than $\Z$.
This creates one of those situations where, although the full conjecture cannot be mathematically established, partial results make its negation more and more implausible (as is the case for the $\theta(p_c)=0$ conjecture in Bernoulli percolation).
In the present context, the above theorem says that, for the analogous process on $\Z$, if the infection survives then \emph{every} infinite infection path has to visit sites that are arbitrarily far from the origin \emph{in both directions}.
\end{remark}

We briefly mention some challenges faced in proving Theorem~\ref{thm:helpout}.
Since this is no longer in the regime $ \lambda_e \leq \lambda_i $, the process under consideration is not attractive.
So there is no simple way to compare different configurations.
As an anecdote, we had previously obtained several ``proofs'' for Theorem~\ref{thm:helpout} that turned out to have subtle flaws in them.
In order to prove that the process with parameters $ \lambda_i = \lambda_c $ and $ \lambda_e = \lambda_c + \varepsilon $ survives, we compare it with a process with parameters $ \lambda_e = \lambda_i = \lambda_c $, which is attractive and has zero speed (Proposition~\ref{prop:critinvspeedzero}).
Due to lack of attractiveness, we cannot use subadditivity to study growth speed either, and, in fact, it is not known whether the process has an asymptotic speed at all.
To overcome these problems, we perform a kind of ``restart'' each time an extra infected site is added to the former process.
At each such restart, we perform a thinning of the previous process, obtaining a configuration distributed as the invariant measure for the critical process seen from the rightmost infected site~\cite{CoxDurrettSchinazi91}, the domination between the concerned measures being provided by Proposition~\ref{prop:pathdomination}.
This sequence of restarted critical processes is carefully constructed so as to make it stationary on the one hand, and in some sense comparable with the invariant measure on the other hand.
From stationarity, we can conclude that this auxiliary process has a speed a.s.,\ and that the speed is positive with non-zero probability.
From there we conclude that the non-restarted process has positive lower speed with positive probability, which fortunately is enough to conclude that the process survives.

The remaining sections are organized as follows.
Section~\ref{sec:notation} introduces the notation and some basic
definitions.
Section~\ref{sec:speedetc} contains results concerning the asymptotic speed in the attractive regime $\lambda_e \leq \lambda_i$.
In Section~\ref{sec:edge} we obtain a series of results for the process starting from semi-infinite configurations seen from the boundary, some of which may be of independent interest.
In Section~\ref{sec:critnospeed} we show that the asymptotic speed is $0$ when the pair $(\lambda_i, \lambda_e)$ is critical.
In Sections~\ref{sec:speedincreasese} and~\ref{sec:intinc}, we complete the proof of Theorems~\ref{thm:noverticals} and~\ref{thm:nohorizontals}.
Section~\ref{sec:nonattractive} is devoted to the proof of Theorem~\ref{thm:helpout}; it uses results from Section~\ref{sec:edge} but it is otherwise independent of the other sections as it considers the non-attractive regime.
Section~\ref{sec:survn} expands the arguments of Section~\ref{sec:nonattractive} to prove Remark~\ref{rmk:abitmore}.
Section~\ref{sec:zhang} is independent of all the others and gives a proof of Theorem~\ref{thm:fixenhancement}.
Section~\ref{sec:openproblems} states some open problems.

\section{Terminology, definitions and notation}
\label{sec:notation}

Configurations are
subsets of $ \Z $
represented by capital letters $ A $ and $ B $ if they are fixed, or $\eta, \zeta, \xi$ if they are random.
The set of configurations is denoted $ \Sigma = \mathcal{P}(\Z) \cong \{0,1\}^\Z $.
The set of semi-infinite configurations (only to the left) is denoted $ \Sigma^\ominus $.
The shift to the right is defined as $ T A = A + 1 \subseteq \Z $.
For $ A \in \Sigma $, we define $ \cR A = \sup A $ and $ \cL A = \inf A $.
(On $\Sigma$ and its subspaces, we consider the topology corresponding to the product topology of $\{0,1\}^\Z$.)




Throughout this paper we use an enlarged version of the standard graphical construction of the contact process, which can handle all values of $ \lambda_i $, $ \lambda_e $ and $ \varepsilon $ simultaneously.
The relevant features of this graphical construction are that each space-time point $ (x,t) $ can have a \emph{recovery mark} and each oriented edge from $ (x,t) $ to $ (x \pm 1, t) $ can be \emph{$ \lambda $-open} for a given value of $ \lambda $, in a way that $ \lambda $-open edges are $ \lambda' $-open for all $ \lambda' \geq \lambda $.
We also have an extra clock of rate $\varepsilon$ which is not attached to any specific site (it can be used to account for an increase in $\lambda_e$).
We let $ \Pb $ denote an underlying probability measure in a space where these elements are defined.

Here is a possible implementation.
Sample for each $ x\in\Z $ a PPP (Poisson Point Process) $ \omega_x \subseteq (0,+\infty) $ of intensity $ 1 $.
We say that there is a recovery mark at $ (x,t) $ if $ t \in \omega_x $.
For each pair $ (x , x \pm 1) $, sample a PPP $ \omega_{x,x\pm 1} \subseteq (0,+\infty) \times (0,+\infty) $ with intensity $ 1 $, and we say that the oriented edge going from $ (x,t) $ to $ (x \pm 1,t) $ is $ \lambda $-open if $ \omega_{x,x\pm 1} \cap \big(\{t\} \times (0,\lambda]\big) \ne \emptyset $.
Finally sample a PPP $ \omega_\varepsilon \subseteq (0,+\infty) $ of intensity $ \varepsilon $, and say that there is an \emph{$ \varepsilon $-mark} at time $ t $ if $ t \in \omega_\varepsilon $.

Given $ A\subseteq \Z $, a pair $ (\lambda_i,\lambda_e) $, and $ s \geq 0 $, we construct the process $ (\eta^A_{s,t})_{t\geq s} $ as follows.
At time $ s $, $ \eta^A_{s,s}=A $.
The process jumps from $ B $ to $B\setminus \{x\}$ at time $ t $ if $ x \in B $ and $t \in \omega_x$;
it jumps to $B\cup \{x+1\}$ if $x \in B$ and either $ x \neq \mathcal R B$ and the edge from $(x,t)$ to $(x+1,t)$ is $\lambda_i$-open or $ x=\mathcal R B $ and the edge from $(x,t)$ to $(x+1,t)$ is $\lambda_e$-open;
it jumps to $B\cup \{x-1\}$ if $x\in B$ and either $x\neq \mathcal L B$ and the edge from $(x,t)$ to $(x-1,t)$ is $\lambda_i$-open or $x=\mathcal L B$ and the edge from $(x,t)$ to $(x-1,t)$ is $\lambda_e$-open.
These rules a.s.\ define $\eta^A_{s,t}$ for all $t\geq s$ when $A$ is finite.
For infinite $A$, we define $\eta^A_{s,t}$
by taking a limit $ A_n \uparrow A $.
When the parameters need to be specified, we write $ \eta^A_{s,t,\lambda_i,\lambda_e} $.
Finally, if $ s=0 $, we may omit it and write $ \eta^A_{t} $.

For given $A \in \Sigma$, this defines a process $ (\eta^A_t)_{t\geq 0} $ started from configuration $\eta_0 = A$, informally described in the previous section.
Process
$ (\eta^-_t)_t $ starts from $ \eta_0 = \Z_- $,
$ (\eta^+_t)_t $ starts from $ \eta_0 = \Z_+ $,
$ (\eta^x_t)_t $ starts from $ \eta_0 = \{x\} $,
$ (\eta^0_t)_t $ starts from $ \eta_0 = \{0\} $,
and
$ (\eta^\nu_t)_t $ starts from random $ \eta_0 $ with distribution $ \nu $.
We reserve the letter $ \mu $ for the distribution provided by Proposition~\ref{prop:existsinv}.

Processes denoted by the letter $\eta$ always evolve with parameters $ (\lambda_i,\lambda_e) $ and always use the same graphical construction, differing only in the starting time and initial configuration.
Processes with different rules will be denoted with letters $ \xi $ or $ \zeta $, depending on the section.
In a slight abuse of notation, the superscript in $ \xi^1_t,\xi^2_t,\dots,\xi^n_t,\dots $ may be used to index a sequence of processes rather than specify the initial condition.

\begin{remark}
\label{rmk:allornothing}
$ \Pb(\eta^A_t \ne \emptyset \ \forall t) $ is either positive for all finite non-empty $ A $ or zero for all finite non-empty $ A $.
\end{remark}

The above construction defines the processes of interest.
However, when analyzing these processes, we shall make use of the notions of paths, open paths, active paths, and rightmost paths.
A \emph{path} is a discrete càdlàg function $ \gamma : [t_0,t_1] \to \Z $ 
whose jumps are of length $1$.
We say that a path is \emph{$\lambda$-open} if all its jumps correspond to \emph{$\lambda$-open} edges and it does not cross any recovery marks.
Let $(\lambda_i,\lambda_e)$ be fixed.
Given $x\in \Z$, $A \subseteq \Z$ and $s\geq 0$, an \emph{active path} starting from a space-time point $(x,s)$ for configuration $A$ is a path that does not cross any recovery marks, and whose jumps happen at $\lambda_i$-open or $\lambda_e$-open edges, depending on whether the jump is internal or external to $[\cL \eta^A_{s,t},\cR \eta^A_{s,t}]$.
Due to the nearest-neighbor property of the jumps, if there is an active path starting at $A \times \{s\}$ and ending at $B \times \{t\}$ for a given initial configuration $ A' \supseteq A $, then there exists a rightmost active path going from $A$ to $B$ in the same time interval for the same initial configuration $A'$.

\section{Attractiveness and asymptotic speed}
\label{sec:speedetc}

We start this section with a basic lemma:
\begin{lemma}
\label{lemma:attractive}
On the octant $ \{\lambda_i \geq \lambda_e \geq 0 \} $,
the configuration $ \eta^A_{s,t,\lambda_i,\lambda_e} $ is increasing in $ A $, $ \lambda_i $ and $ \lambda_e $ for fixed $ t \geq s \geq 0 $.
\end{lemma}

The lemma follows immediately from the construction described in \S\ref{sec:notation}, since infecting more sites can only affect the infection rates positively, by turning some external edges into internal ones.

\begin{remark}
The process is not attractive when $\lambda_e > \lambda_i$.
To see this, consider $A=\{0\}$ and $B=\{3\}$, then for small values of $t$ we have 
$\Pb(1 \in \eta^A_t)>\Pb(1 \in \eta^{A\cup B}_t)$.
\end{remark}

\begin{remark}
When $\lambda_e<\lambda_i$, the process is attractive but not additive. Indeed, with the same $A$ and $B$ as above, we have
$\Pb(1\in \eta^{A\cup B}_t) > \Pb(1\in \eta^A_t)+\Pb(1\in \eta^B_t)$ for small $t$,
hence there is no coupling such that
$\eta^{A\cup B}_t= \eta^A_t \cup \eta^B_t$,
even though 
$\eta^{A\cup B}_t \supseteq  \eta^A_t \cup \eta^B_t$.
\end{remark}

Since many properties of the contact process are usually derived from its additivity, we have to modify or reinvent some of their proofs.
Fortunately, for $\lambda_i \geq \lambda_e$, attractiveness is enough to use the Subadditive Ergodic Theorem as in~\cite[Chapter~VI]{Liggett05}.
This allows us to work with the asymptotic speed of the boundary of the process.
With this in mind, we define
\begin{equation}
\nonumber
\alpha_t(\lambda_i,\lambda_e) = {\E [\cR \eta^-_t]}
\end{equation}
and 
\begin{equation}
\nonumber
\displaystyle \alpha(\lambda_i,\lambda_e) = \inf\limits_{t>0} \frac{\alpha_t(\lambda_i,\lambda_e)}{t}.
\end{equation}

\begin{lemma}
\label{lemma:speed}
Let $0 < \lambda_e \leq \lambda_i < \infty$ be fixed.
Then,
\[
\frac{\cR \eta^-_t}{t}
\overset{\text{a.s.}}{\longrightarrow}
\alpha(\lambda_i,\lambda_e)
=
\lim_{t\to\infty}
\frac{\alpha_t(\lambda_i,\lambda_e)}{t}
.
\]
\end{lemma}
Thanks to Lemma~\ref{lemma:attractive},
the proof is the same as in~\cite[VI.2.19]{Liggett05}.

\begin{lemma}
\label{lemma:gspeed}
Suppose $1<\lambda_e \leq \lambda_i<\infty$ and $\theta(\lambda_i,\lambda_e) > 0$.
Let $ B \in \Sigma^\ominus $.
Then,
\[
\frac{\cR \eta^B_t}{t}
\overset{\text{a.s.}}{\longrightarrow}
\alpha(\lambda_i,\lambda_e) 
.
\]
\end{lemma}
\begin{proof}
Write $ \alpha = \alpha (\lambda_i,\lambda_e) $.
First note that, if $ \eta^0_t \ne \emptyset $, then $ \cR \eta^-_t = \cR \eta^0_t $ (this holds because both processes are defined from the same graphical construction and infections are to nearest-neighbor only).
Hence, a.s., if $ \eta^x_t \ne \emptyset $ for all $ t \geq 0 $, then $ \lim_t t^{-1} \cR \eta^x_t = \alpha $ by the previous lemma.
Now since $\theta(\lambda_i,\lambda_e) > 0$ and $ B $ has infinitely many points, a.s.\ there exists $ x \in B $ such that $ \eta^x_t \ne \emptyset $ for all $ t \geq 0 $, which implies that
$ \liminf_t t^{-1} \cR \eta^B_t \geq \lim_t t^{-1} \cR \eta^x_t = \alpha $ because $ \{x\} \subseteq B $.

On the other hand, we can assume $ B \subseteq \Z_- $ without loss of generality, and from this we get $ \limsup_t t^{-1} \cR \eta^B_t \leq \lim_t t^{-1} \cR \eta^-_t = \alpha $, a.s.
\end{proof}

\begin{lemma}
\label{lemma:rightcontinuous}
The function $\alpha(\cdot,\lambda_e)$ is right continuous on $[\lambda_e,\infty)$,
and
the function $\alpha(\lambda_i,\cdot)$ is right continuous on $[0,\lambda_i]$.
\end{lemma}

\begin{proof}
For all $ t > 1 $, the function $ t^{-1} \alpha_t(\cdot,\lambda_e) $ is continuous.
Hence, $ \alpha(\cdot,\lambda_e) $ is upper semi-continuous.
By Lemma~\ref{lemma:attractive}, it is non-decreasing, and therefore it is right-continuous.
The right continuity of $ \alpha(\lambda_i,\cdot) $ is proved in exactly the same way.
(Cf.~\cite[VI.2.27b]{Liggett05}.)
\end{proof}

\begin{lemma}
\label{lemma:nonnegative}
Suppose $\lambda_e \leq \lambda_i$ and $\theta(\lambda_i,\lambda_e) > 0$.
Then $ \alpha(\lambda_i,\lambda_e) \geq 0$.
\end{lemma}
\begin{proof}
By Lemma~\ref{lemma:attractive}, $ \eta^{0}_t \subseteq \eta^-_t $ and $ \eta^{0}_t \subseteq \eta^+_t $ for all $ t $.
If $ \alpha(\lambda_i,\lambda_e) < 0 $, by Lemma~\ref{lemma:gspeed}, we have $ \eta^-_t \cap \eta^+_t = \emptyset $ for all large $ t $, almost surely, which by the previous sentence implies that $\eta^0_t = \emptyset$ for large $t$, and therefore $ \theta(\lambda_i,\lambda_e)=0 $.
(Cf.~\cite[VI.2.27a]{Liggett05}.)
\end{proof}

\begin{theorem}
\label{thm:critnospeed}
For $1< \lambda_e \leq \lambda_c $, we have that
$\alpha(\lambda_*^i(\lambda_e),\lambda_e)=0$.
For $ \lambda_i \geq \lambda_c $, we have
$\alpha(\lambda_i,\lambda_*^e(\lambda_i))=0$.
\end{theorem}

We give the proof in \S\ref{sec:critnospeed}.

\begin{proposition}
\label{prop:speedincreasese}
Suppose $\theta(\lambda_i,\lambda_e) > 0$, and $ \lambda_e + \varepsilon \leq \lambda_i $.
Then
$
\alpha(\lambda_i,\lambda_e+\varepsilon) \geq \alpha(\lambda_i,\lambda_e) + \varepsilon
.$
\end{proposition}

We give the proof in \S\ref{sec:speedincreasese}.

Recall Theorem~\ref{thm:ds} which says that
$\lambda_*^i(\lambda_e)<\infty$ for $\lambda_e>1$.

\begin{corollary}
\label{cor:speedincrease}
Let $ 1 < \lambda_e < \lambda_e' \leq \lambda_c $.
Then $\alpha(\lambda_*^i(\lambda_e),\lambda_e')>0.$
\end{corollary}

\begin{proof}
Write $\varepsilon=\lambda_e'-\lambda_e$.
Take $\lambda_i = \lambda_*^i(\lambda_e) \geq \lambda_c \geq \lambda_e' $ (note that $\lambda_*^i(\lambda_e) < \infty$ by Theorem~\ref{thm:ds}).
For $\delta > 0$, we have $\theta(\lambda_i+\delta,\lambda_e) > 0$.
By Proposition~\ref{prop:speedincreasese} and Lemma~\ref{lemma:nonnegative}, $\alpha(\lambda_i+\delta,\lambda_e') \geq \varepsilon$.
Letting $\delta \downarrow 0$,
by Lemma~\ref{lemma:rightcontinuous} we get
$\alpha(\lambda_i,\lambda_e') \geq \varepsilon$.
\end{proof}

\begin{proposition}
\label{prop:speedincint}
Let $ \lambda_i' > \lambda_i > \lambda_c $.
Then $\alpha(\lambda_i' ,\lambda_*^e(\lambda_i))>0.$
\end{proposition}

We give the proof in \S\ref{sec:intinc}.

\begin{proof}
[\proofname\ of Theorems~\ref{thm:noverticals} and~\ref{thm:nohorizontals}]
Suppose that $1 < \lambda_e < \lambda_e' \leq \lambda_c $ and take $\lambda_i = \lambda_*^i(\lambda_e)$ (note that $\lambda_i \geq \lambda_c>0$ and, by Theorem~\ref{thm:ds}, $\lambda_*^i(\lambda_e) < \infty$).
By Corollary~\ref{cor:speedincrease}, $\alpha(\lambda_i,\lambda_e')>0$.
By Theorem~\ref{thm:critnospeed}, $\lambda_i > \lambda_*^i(\lambda_e')$, since otherwise we would have $\alpha(\lambda_i,\lambda_e')=0$.
So $\lambda_*^i(\lambda_e) > \lambda_*^i(\lambda_e')$, proving
Theorem~\ref{thm:noverticals}.

Note that
$\lambda_*^e(\cdot)$ is non-increasing on
$[\lambda_c,+\infty)$
by attractiveness, since $\lambda_*^e(\cdot) \leq \lambda_c$ on this interval.
Let $ \lambda_i' > \lambda_i > \lambda_c $ and take $\lambda_e = \lambda_*^e(\lambda_i)$ (note that $1\leq \lambda_e \leq \lambda_c$).
By Proposition~\ref{prop:speedincint}, $\alpha(\lambda_i',\lambda_e)>0$.
By Theorem~\ref{thm:critnospeed}, $\lambda_e > \lambda_*^e(\lambda_i')$, since otherwise we would have $\alpha(\lambda_i',\lambda_e)=0$.
So $\lambda_*^e(\lambda_i) > \lambda_*^e(\lambda_i')$, proving
Theorem~\ref{thm:nohorizontals}.
\end{proof}

\section{The process seen from the boundary}
\label{sec:edge}
In this section we obtain a series of properties for the process starting from semi-infinite configurations, seen from the boundary.
These results will be used in subsequent sections and some of them may be of independent interest.

We start with some more notation.
Define $ \Sigma^\odot = \{ A \in \Sigma^\ominus: \cR A = 0 \} $ and define the map $ \Psi : \Sigma^\ominus \to \Sigma^\odot $ by $ \Psi A = T^{-\cR A} A $.
For finite non-empty $A$, we define $\Psi A$ by the same formula, and we define $\Psi \emptyset = \emptyset$.
We define the addition of a site
$ \Xi : \Sigma^\ominus \to \Sigma^\ominus $
by
$ \Xi \eta = \eta \cup \{\cR\eta+1\} $.
A real-valued function $f$ on $ \Sigma $ will be called \emph{increasing} if $f(A) \leq f(B)$ whenever $A \subseteq B$.
If $\eta $ and $\xi$ are random elements of $\Sigma$, we say that $\eta$ stochastically dominates $\xi$, denoted $ \eta \succcurlyeq \xi $, if $\E [f(\eta)]\geq \E [f(\xi)]$ for every bounded increasing measurable function $f$.

\begin{proposition}
\label{prop:existsinv}
Suppose $\lambda_e \leq \lambda_i$ and $\theta(\lambda_i,\lambda_e) > 0$, or $ \lambda_i=\lambda_e=\lambda_c $.
Then
there exists a measure~$ \mu $ supported on $ \Sigma^\odot $ such that $ \Psi\eta^\mu_t \sim \mu $ for every $ t\geq 0 $, and $ \Psi \eta^A_t \to \mu $ weakly for every fixed $ A \in \Sigma^\odot $.
\end{proposition}

Throughout this paper, the letter $ \mu $ refers to the above measure.
The case $\lambda_e=\lambda_i=\lambda_c$ is~\cite[Theorem~1]{CoxDurrettSchinazi91}.
So we have to prove Proposition~\ref{prop:existsinv} for $ \lambda_e \leq \lambda_i $ and $ \theta(\lambda_i,\lambda_e)>0 $.
We will use the following lemma.

\begin{lemma}
\label{lemma:domination}
Let $ \lambda_i > \lambda_c $ and $ \lambda_e \geq 0 $.
Then for all $ t \geq 0 $,
\[
\Psi \eta^-_t \succcurlyeq \zeta \cap \Z_-
,
\]
where $ \zeta $ denotes a random configuration distributed as the upper invariant measure for the standard contact process with parameter $ \lambda_i > \lambda_c $.
\end{lemma}
\begin{proof}
Fix $ t \geq 0 $.
Let $ \Gamma $ denote the rightmost active path for the initial configuration $ \Z_- $ in the time interval $[0, t] $.

For a deterministic path $(\gamma_s)_{s\in[0,t]}$,
let
$D_{\gamma}^\ominus = \{(x,s) \in \R \times [0,t] : x < \gamma_s \wedge \gamma_{s-} \}$, $D_{\gamma}^\odot = \{(x,s) \in \Z \times [0,t] : s=0 \text{ or } x = \gamma_s \wedge \gamma_{s-} \}$
and
$D_\gamma^{\oplus} = \{(x,s) \in \R \times [0,t] : x \geq \gamma_s \wedge \gamma_{s-} \}$.
So the path $\gamma$ is entirely inside $D_\gamma^{\oplus}$, but edges that connect $\gamma$ with $D_\gamma^{\ominus}$ are inside $D_\gamma^{\ominus}$.

On the event $\{\Gamma=\gamma\}$, the configuration $\eta^-_t$ is given by the set of sites that can be reached from $D_{\gamma}^\odot$ within $D_{\gamma}^\ominus$ via a $ \lambda_i $-open path, plus $\gamma_t$ itself (see Figure~\ref{fig:rightmost}).
On the other hand, the event $\{\Gamma=\gamma\}$ is determined by the graphical construction $\omega$ on the region ${{D_\gamma^{\oplus}}}$ (see Figure~\ref{fig:rightmost}).
Hence, the conditional distribution of $\eta_t^-$ given $\{\Gamma=\gamma\}$ is the distribution of the set of sites that can be reached from $D_{\gamma}^\odot$ within $D_{\gamma}^\ominus$, without conditioning.
This has the same distribution as the set of sites $x \leq \gamma_t$ such that there is a $\lambda_i$-open path starting at $(x,t)$, going backwards in time, and ending at either $\Z \times \{0\}$ or at a point of $\gamma$.
And since we are interested in the conditional distribution of $\Psi \eta^-_t$, we just need to translate $\gamma$ horizontally so that $\gamma_t=0$.

\begin{figure}[b!]
\hfil
\includegraphics[page=1,width=.95\textwidth]{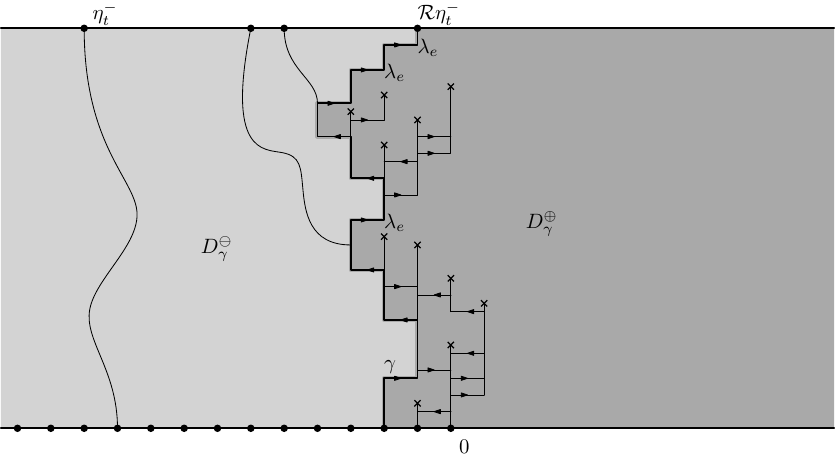}
\caption{%
Starting from the semi-infinite configuration $\eta_0=\Z_-$, the rightmost active path $\Gamma$ does not change if the graphical construction is changed to the left of it. This is because the event $\{\Gamma=\gamma\}$ is determined by the following conditions: $\gamma$ does not contain recovery marks, and its jumps happen on $\lambda_i$-open edges, except for the jumps taking place in times when no other active path exists to the right, in which case they must be $\lambda_e$-open, and active paths to the right of $\gamma$ will neither reach time $t$ nor connect back to $\gamma$ (otherwise they themselves would have been the rightmost active path). The configuration $\eta_s^-$ always has infinitely many infected sites to the left a.s.,\ therefore it does not have a leftmost site where the infection rate to the left would have been $\lambda_e$.
}
\label{fig:rightmost}
\end{figure}


Finally, notice that the random configuration $\zeta$ is distributed exactly as the collection of sites $x \in \Z$ such that there is an infinite $\lambda_i$-open path starting at $(x,t)$ and going backwards in time.
When we consider $\zeta \cap \Z_-$ instead, this is the same as restricting to sites $x \leq 0$.
Since it is easier to connect $(x,t)$ to $\Z_- \times \{0\}$ than to find an infinite backwards path, we conclude that $\Psi \eta^-_t$ dominates $\zeta \cap \Z_-$, which proves the lemma.
\end{proof}

\begin{proof}
[\proofname\ of Proposition~\ref{prop:existsinv}]
We adapt the argument in~\cite[\S3]{GalvesPresutti87}.
Consider the law of $ \Psi\eta^-_t $ at times $ t \in \N $.
Since $ \Sigma $ is a compact metric space, there is a subsequence $ (t_k)_k $ such that $ t_k \to +\infty $ and $ \Psi\eta^-_{t_k} \to \mu $ weakly as $ k\to\infty $, for some probability measure $ \mu $ on $ \Sigma $.

The proposition does not follow immediately from this, because the space $\Sigma^\odot \subseteq \Sigma$ is not itself compact, and moreover the transition probabilities of the process $ (\Psi\eta_t)_t $ are discontinuous at finite configurations.
That is, for fixed time $t>0$, in order to approximate transition probabilities after $t$ units of time, one may need to look arbitrarily far, and approximating a finite configuration by infinite configurations in the product topology can yield different distributions at time $t$.

However, it is still true that, for every $ t\geq 0 $ and every continuous bounded $ f:\Sigma \to\R $, the functional $ S_t f $ given by $ A \mapsto \E[f(\Psi\eta^A_t)] $ is continuous at infinite configurations $ A \in \Sigma^\odot $.

The construction in~\cite[\S3]{GalvesPresutti87} gives
\begin{align}
\label{eq:a1}
\Psi \eta^A_t \to \mu \text{ weakly as } t\to\infty, \text{ for every } A \in \Sigma^\odot
.
\end{align}
The argument is essentially that, since $\theta(\lambda_i,\lambda_e) > 0$, after some time the process will forget the initial configuration and will converge to the same $\mu$ for every $A$.
We refer to~\cite{GalvesPresutti87} for the details, keeping in mind that the assumption $ \lambda_e \leq \lambda_i $ provides attractiveness which is used there.

By~\eqref{eq:a1} and Lemma~\ref{lemma:domination}, we conclude that
\begin{align}
\label{eq:a2}
\mu(\Sigma^\odot) = 1
.
\end{align}
Since $ S_t f $ is continuous on $ \Sigma^\odot $ for every continuous $ f:\Sigma\to\R $, 
\eqref{eq:a1} and~\eqref{eq:a2} imply that $ \Psi\eta^\mu_t \sim \mu $ for every $ t \geq 0 $.
This concludes the proof of the proposition.
\end{proof}

\begin{proposition}
\label{prop:speedconstant}
Suppose $\lambda_e \leq \lambda_i$ and $\theta(\lambda_i,\lambda_e) > 0$.
Then
\begin{equation}
\nonumber
\E [\cR \eta^\mu_t] = \alpha(\lambda_i,\lambda_e) \cdot t
\end{equation}
for all $ t \geq 0 $, where $ \mu $ is given by Proposition~\ref{prop:existsinv}.
\end{proposition}
\begin{proof}
Note that $ \cR \eta^\mu_t $ has stationary increments.
Thus, $ t^{-1} \E [\cR \eta^\mu_t] = \E [\cR \eta^\mu_1] \in [-\infty,\lambda_e]$ and, by the Ergodic Theorem, $ n^{-1} \cR \eta^\mu_n \to V $ a.s.,\ for (possibly random) $ V $ with $ \E V = \E [\cR \eta^\mu_1] $.
On the other hand, by Lemma~\ref{lemma:gspeed}, $ n^{-1} \cR \eta^\mu_n \to \alpha(\lambda_i,\lambda_e) $ a.s.,\ concluding the proof.
\end{proof}


\begin{proposition}
\label{prop:pathdomination}
Let $ \lambda_e \leq \lambda_i $ and suppose $\theta(\lambda_i,\lambda_e) > 0$, or that $ \lambda_e=\lambda_i=\lambda_c $.
If $ \zeta $ denotes a random configuration in $ \Sigma^\odot $ with law $ \mu $, then $ \Xi\zeta \succcurlyeq T\zeta $.
\end{proposition}

\begin{proof}
We will work with a finite time $t$ and later apply Proposition~\ref{prop:existsinv}.
Fix $ t >0 $.
Let $ \Gamma $ denote the rightmost active path connecting $ \eta^-_0 $ at time $ 0 $ to $ \eta^-_t $ at time $ t $.
As seen in the proof of Lemma~\ref{lemma:domination}, the conditional distribution of $\eta^-_t$ given the event $\{\Gamma = \gamma\}$ is the same as the unconditioned distribution of the set of sites $x \leq \gamma_t$ such that there is a $\lambda_i$-open path starting at $(x,t)$, going backwards in time, and ending at either $\Z \times \{0\}$ or at a point of $\gamma$ (the site $x=\gamma_t$ trivially satisfies this condition because $(x,t)$ is already in $\gamma$).
Denote this distribution by $\nu^{\gamma}$.

Define the re-centered curve $(\bar\gamma_s)_{s\in[0,t]}$ by $\bar\gamma_s = \gamma_s - \gamma_t$.
Since $\Gamma_t = \cR \eta^-_t $, by translation invariance, the conditional distribution of $ \Psi\eta^-_t $ given that $\{\Gamma=\gamma\}$ equals $\nu^{\bar\gamma}$.

We now consider the configuration $ \Xi\eta^-_t $, which equals $ \eta^-_t \cup \{\cR \eta^-_t +1\} $.
To give a similar description as before, we define the path $ (\gamma^+_s)_{s\in[0,t]} $ given by
\[
\gamma^+_s =
\begin{cases}
\gamma_s ,& s \in [0,t), \\
\gamma_t+1 ,& s=t.
\end{cases}
\]
Note that
the conditional distribution of $ \Xi\eta^-_t $ given $\{\Gamma=\gamma\}$ is the same as the unconditioned distribution of the set of sites $x \leq \gamma_t+1$ such that there is a $\lambda_i$-open path starting at $(x,t)$, going backwards in time, and ending at either $\Z \times \{0\}$ or at a point of $\gamma^+$ 
(the sites $x=\gamma_t$ and $x=\gamma_t+1$ trivially satisfy this condition because the corresponding points $(x,t)$ are already in $\gamma^+$).

As before, we define the re-centered curve $(\bar\gamma^+_s)_{s\in[0,t]}$ by $\bar\gamma^+_s = \gamma^+_s - \gamma^+_t$.
Since $\gamma^+_t = \cR( \Xi\eta^-_t)$, by translation invariance, the conditional distribution of $ \Psi \Xi \eta^-_t$ given that $\{\Gamma=\gamma\}$ equals $\nu^{\bar\gamma^+}$.

To obtain domination, notice that ${\bar\gamma^+} \leq {\bar\gamma}$, that is, the path ${\bar\gamma^+}$ is to the left of the path ${\bar\gamma}$.
Indeed, ${\bar\gamma^+_t} = {\bar\gamma_t} = 0$ and ${\bar\gamma^+_s} = {\bar\gamma_s} - 1$ for $ s \in [0,t) $.
So for $x \leq 0$, it is easier to find a backward $\lambda_i$-open path from $(x,t)$ to ${\bar\gamma^+}$ than from $(x,t)$ to ${\bar\gamma}$.
Hence, $\nu^{\bar\gamma^+} \succcurlyeq \nu^{\bar\gamma}$.
Integrating over $\gamma$ we get $ \Psi \Xi \eta^-_t \succcurlyeq \Psi \eta^-_t $.
Since $ \Xi \Psi = T \Psi \Xi $, this gives $ \Xi \Psi \eta^-_t \succcurlyeq T \Psi \eta^-_t $.

We now let $t \to \infty$.
By Proposition~\ref{prop:existsinv}, $ \Psi\eta^-_t \to \zeta $ in distribution.
Since $\Xi$ and $ T $ are continuous and preserve monotonicity, this gives $\Xi \zeta \succcurlyeq T \zeta$, proving the proposition.
\end{proof}

We now move to the last result of this section.
For the classical contact process with critical parameter $\lambda = \lambda_c$ (where it is known that $\alpha=0$), we would like a statement similar to Lemma~\ref{lemma:speed}, regarding both the mean and almost-sure behavior of the position of the rightmost infected site, starting from the invariant measure $\mu$.
However, Lemma~\ref{lemma:speed} is proved using subadditivity and only works for the extreme initial configuration $\eta_0 = \Z_-$.
In fact, since $\theta(\lambda_c,\lambda_c)=0$, one can easily construct infinite initial configurations $A \in \Sigma^\ominus$ for which ${t^{-1}}\,{\cR \eta_t^A} \to -\infty$ a.s.\ as $t\to\infty$, by taking $A$ sparse enough (this in contrast with how Proposition~\ref{prop:speedconstant} was obtained as an immediate consequence of Lemma~\ref{lemma:gspeed}).

The following proposition answers this question about the classical contact process at criticality, and we believe it is of independent interest.

\begin{proposition}
\label{prop:critinvspeedzero}
For $\lambda_i=\lambda_e=\lambda_c$,
\(
\lim\limits_{t \to \infty} \frac{\cR \eta^\mu_t}{t}
=
\E[\cR \eta_1^\mu]
=
\alpha(\lambda_c,\lambda_c)
\)
a.s.
\end{proposition}

\begin{proof}
As in the proof of Proposition~\ref{prop:speedconstant}, 
there is $ \beta \in [-\infty,\lambda_c] $ such that $ \E[\cR\eta^\mu_t] = \beta t $ for all $ t \geq 0 $ and
$
\frac{\cR \eta^\mu_t}{t} \to
V
$
a.s.,\ where $ V $ is such that $ \E V = \beta $.
However, here we cannot use Lemma~\ref{lemma:gspeed}.
Below we will show that $ V = \beta $ a.s.\ and that
$ \beta=\alpha $, where $ \alpha = \alpha(\lambda_c,\lambda_c) = 0 $.

Since $ \eta^\mu_t \subseteq \eta^-_t $,
we have
\[
V=
\lim_t \frac{\cR \eta^\mu_t}{t}
\leq \lim_t \frac{\cR \eta^-_t}{t}
=
\alpha
\text{  \ a.s.}
,
\]
where the last limit holds by Lemma~\ref{lemma:speed}.
So it is enough to prove that $ \beta \geq \alpha $.

For $ A \in \Sigma^\ominus $, let $ \cS A \in \N $ denote the distance between the rightmost and second rightmost site in $ A $.
We first note that
\[
\tfrac{\dd}{\dd t} \E[\cR\eta^\nu_t] = \lambda_c - \E[\cS \eta^\nu_t]
\]
for every measure $ \nu $ on $ \Sigma^\odot $.
More precisely,
if $ \E|\cR\eta^\nu_t| < \infty $ and $ \E[\cS \eta^\nu_t] < \infty $, then the above equation holds,
and
if $ \E[\cR\eta^\nu_t] = -\infty $ or $\E[\cS \eta^\nu_t]=+\infty$ then $ \E[\cR\eta^\nu_s] = -\infty $ for all $ s > t $.
This follows from the fact that $\cR\eta_t$ jumps by $+1$ at rate $\lambda_c$ and jumps by $- \cS \eta_t$ at rate $1$.

In particular, since $ \E[\cR\eta^\mu_t] = \beta t $ for every $t \geq 0$, we have that
\[
\lambda_c - \E[\cS \eta^\mu_t] = \beta
\]
for every $t \geq 0$.

On the other hand, by Proposition~\ref{prop:existsinv}, $ \Psi \eta^-_t {\to} \mu $ in distribution as $ t \to \infty $, hence
\[
\liminf_t
\E[\cS \eta^-_t]
\geq
\E[\cS \eta^\mu_t]
\]
by Fatou's lemma. (We believe $\E[\cS \eta^-_t] \to \E[\cS \eta^\mu_t]$ but the above inequality will be enough.)

Estimating the mean drift for the extreme initial configuration, we get
\[
\limsup_t
\tfrac{\dd}{\dd t} \E[\cR\eta^-_t]
=
\lambda_c -
\liminf_t
\E[\cS \eta^-_t]
\leq
\lambda_c - 
\E[\cS \eta^\mu_t]
=
\beta
.
\]

Finally, by~Lemma~\ref{lemma:speed},
\[
\frac{\E [\cR \eta_t^-]}{t}
\to
\alpha
\]
as $t \to \infty$.
Moreover, for all $t > 0$,
$
\frac{\E [\cR \eta_t^-]}{t}
\geq
\alpha
,
$
and since
$
\frac{\E [\cR \eta_t^-]}{t}
\leq
\lambda_c
$
we conclude that $\E |\cR \eta_t^-|<\infty$.

Combining this with the estimate on the mean drift, we get
\[
\alpha =
\lim_t
\frac{\E[\cR\eta^-_t]}{t}
\leq \beta
,
\]
which concludes the proof.
\end{proof}

\section{Zero speed at criticality}
\label{sec:critnospeed}

In this section we prove Theorem~\ref{thm:critnospeed}.
It is possible to prove this theorem rather easily once it is known that $\theta(\lambda_i, \lambda^e_*(\lambda_i))=0$ and $\theta(\lambda^i_*(\lambda_e), \lambda_e)=0$.
This is proved in Chapter~2 of~\cite{Terra24}, where the proof is an adaptation of~\cite{BezuidenhoutGrimmett90}. A more general result of this nature is available in \cite{BezuidenhoutGray94}. However all these references rely on the rather  involved dynamic renormalization technique. For this reason we provide a different proof.

The general structure of the proof is a classical block argument.
However, the specifics of this model bring a number of complications.
In particular, the building block is more subtle to define and analyze because whether the infection spreads through a given path may depend on the
configuration outside this path.

Assuming that $ \lambda_e \leq \lambda_i $ and $ \alpha(\lambda_i,\lambda_e)>0 $, we will derive a finite condition (a condition that depends on a bounded space-time box in the graphical construction) which in turn implies that $ \theta(\lambda_i,\lambda_e)>0 $.
Since said condition is finite, it is still satisfied for slightly smaller values of the parameters.
This way we can conclude that
$ \lambda_e > \lambda_*^e(\lambda_i) $
and
$ \lambda_i > \lambda_*^i(\lambda_e) $,
and this chain of implications combined with Lemmas~\ref{lemma:rightcontinuous} and~\ref{lemma:nonnegative} proves the theorem.

We call \emph{domain} any subset $ D \subseteq \Z \times \R $ of the form
$$ \{ (x,t) : t_0 \leq t \leq t_1, L(t) \leq x \leq R(t) \} $$
for some $ t_0 \leq t_1 \in \R $ and $ L,R:[t_0,t_1] \to \Z $ such that $ L(t) \leq R(t) $ for all $ t\in (t_0,t_1] $. The \emph{bottom} of $ D $ is given by the region $ \{ (x,t_0) : L(t_0) \leq x \leq R(t_0) \} $, the \emph{top} is defined analogously.

\begin{definition}
[Open branch]
A \emph{$ (\lambda_i,\lambda_e) $-open branch $ \mathcal{O} $ in a domain $ D $} is a collection of paths $ \gamma $ contained in $ D $ such that each path $ \gamma \in \mathcal{O} $ satisfies:
\begin{enumerate}
[$ (i) $]
\item
$ \gamma $ starts at the bottom of $ D $,
\item
$ \gamma $ contains no recovery marks,
\item
there is a $ \lambda_i $-mark at each jump of $ \gamma $,
\item
\label{item:condiv}
there is a $ \lambda_e $-mark at each jump of $ \gamma $ that points in a direction not occupied by other paths in $ \mathcal{O} $.
\end{enumerate}
For each $ t \in (t_0,t_1] $ denote $ \cO_t = \{ x : \exists \gamma\in\cO, \gamma(t)=x\} $. The points $ \cR \cO_t $ and $ \cL \cO_t $ denote the rightmost and leftmost points attained by this branch at time $ t $.
With this notation, condition~\ref{item:condiv} above reads as follows: for each $ \gamma \in \cO $, if a given jump of $ \gamma $ corresponds to a jump of $ \cL\cO_t $ to the left or a jump of $ \cR\cO_t $ to the right, then there is a $ \lambda_e $-mark corresponding to that jump of $ \gamma $.

An open branch is \emph{complete} if it has a path which starts at the bottom of $ D $ and reaches the top of $ D $.
\end{definition}

\begin{figure}[b!]
\hfil
\includegraphics[page=1,width=.95\textwidth]{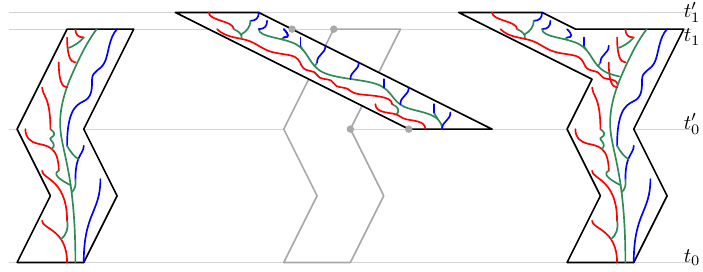}
\caption{%
Open branches in $ D $ (left), $ D' $ (center), and the grafting of $ D' $ onto $ D $ (right).
Blue paths only jump to the right, and are $ \lambda_e $-open.
Red paths only jump to the left, and are $ \lambda_e $-open.
Green paths jump in both directions and are $ \lambda_i $-open.
For each open branch, at each instant in time, the rightmost site of the branch is occupied by a blue path and the leftmost site of the branch is occupied by a red path.
Thin light gray lines indicate $t_0, t_0', t_1, t_1'$.
A grayed version of $D$ behind $D'$ illustrates their relative position, and the four gray dots illustrate conditions
$ R(t_0') \leq L'(t_0') $,
and
$ R'(t_1) \leq L(t_1) $.
\hfill (color online)
}
\label{fig:crossinglemma}
\end{figure}

As a side remark, for $ \lambda_e \leq \lambda_i $, a subset of an open branch need not be an open branch, but the union of two open branches in $ D $ is an open branch.

We now describe a way to concatenate open branches, see Figure~\ref{fig:crossinglemma}.

\begin{definition}
[Crossing domains and grafting]
Given two domains $ D $ and $ D' $ (with associated $ t_0,t_1,L,R $ and $ t_0',t_1',L',R' $, respectively), we say that $ D' $ \emph{crosses $ D $ from right to left} if
$ t_0 \leq t_0' \leq t_1 \leq t_1' $,
$ R(t_0') \leq L'(t_0') $,
and
$ R'(t_1) \leq L(t_1) $.
If $ D' $ crosses $ D $ from right to left, we define the \emph{grafting of $ D' $ onto $ D $} as the domain $ D'' $ defined as follows.
Take
$ t_0''=t_0 $,
$ t_1''=t_1' $,
\[
L''(t) =
\begin{cases}
L(t) , & t_0 \leq t \leq t_0'
,
\\
L(t) \wedge L'(t), & t_0' < t < t_1
,
\\
L'(t), & t_1 \leq t \leq t_1'
,
\end{cases}
\]
and
\[
R''(t) =
\begin{cases}
R(t) , & t_0 \leq t \leq t_0'
,
\\
R(t) \vee R'(t), & t_0' < t < t_1
,
\\
R'(t), & t_1 \leq t \leq t_1'
.
\end{cases}
\]
See Figure~\ref{fig:crossinglemma}.
\end{definition}

\begin{lemma}
[Grafting lemma]
\label{lemma:crossing}
Suppose a domain $ D' $ crosses a domain $ D $ from right to left.
For $ \lambda_e \leq \lambda_i $, if there are complete open branches in $ D $ and in $ D' $, then there is a complete open branch in the grafting of $ D' $ onto $ D $.
\end{lemma}
\begin{proof}
Let $ \cO $ and $ \cO' $ be complete open branches in $ D $ and $ D' $.
We will describe how to construct a complete open branch $ \cO'' $ in the grafting $ D'' $ of $ D' $ onto $ D $.
A look at Figure~\ref{fig:crossinglemma} may help understand how $ \cO'' $ is constructed.

Let $ t_0 \leq t_0' \leq t_1 \leq t_1' $ denote the start and end times of $ D $ and $ D' $.

Fix some path $ \gamma^* \in \cO $ that starts at the bottom of $ D $ and ends at the top of $ D $.
For each $ \gamma' \in \cO' $, write $ I $ for its time span and define $ \gamma'' $ by taking, for each $ t \in [t_0,t_0'] \cup I $,
\[
\gamma''(t) =
\begin{cases}
\gamma^*(t), & t \leq t_0',
\\
\gamma^*(t) \wedge \gamma'(t), & t \in (t_0',t_1] \cap I,
\\
\gamma'(t), & t \in (t_1,t_1'] \cap I.
\end{cases}
\]
Let $ \cO'' $ be the collection of all $ \gamma'' $ obtained this way, as well as all $ \gamma \in \cO $.

To prove the lemma, it is enough to check that $ \cO'' $ is a complete open branch in $ D'' $.
The key observation is that $ \cL \cO_t'' = \cL \cO_t \wedge \cL \cO_t' $ and $ \cR\cO_t'' = \cR\cO_t $ for every $ t \in [t_0',t_1] $.
Hence, each time $ \cL\cO''_t $ jumps to the left, it corresponds to $ \cL \cO_t $ or $ \cL \cO_t' $ jumping to the left, and in either case there is a corresponding $ \lambda_e $-open edge in the graphical construction.
Likewise, jumps of $ \cR\cO''_t $ to the right correspond to a $ \lambda_e $-open edge.
Other jumps correspond to a jump of a path in $ \cO $ or $ \cO' $, which in turn occur at a $ \lambda_i $-open edge.
For times $ t \leq  t_0' $, we have $ \cO''_t = \cO_t $, and for times $ t > t_1 $ we have $ \cO''_t = \cO'_t $.
It follows that $ \cO'' $ satisfies condition~\ref{item:condiv} above.
The other conditions are more immediate to verify, and we omit the tedious details.
\end{proof}

Given $ \alpha \ne 0 $, $ k>0 $ and $ L \in \N $, define the domain
\[
D_{L,k,\alpha} := \big\{ (x,t) \in \Z\times\R : t \in [0,\tfrac{k L}{|\alpha|}], |x-\alpha t| \leq L \big\}
.
\]

Denote by $ \sE_{L,k,\alpha,\lambda_i,\lambda_e} $ the event that $ D_{L,k,\alpha} $ has a complete $ (\lambda_i,\lambda_e) $-open branch.
Our building block is the following.

We leave $ k $ as a free parameter because the next lemma is already hard to visualize with $ k=\frac{5}{2} $, but later on we will use it with $ k=20 $.

\begin{lemma}
\label{lemma:block}
Suppose $ \lambda_e \leq \lambda_i $ and $ \alpha=\alpha(\lambda_i,\lambda_e)>0 $.
Then, for each $ k\in\N $ fixed, $ \Pb(\sE_{L,k,\alpha,\lambda_i,\lambda_e}) \to 1 $ as
$ L \to \infty,\ L \in 4\N $.
\end{lemma}

\begin{figure}[b!]
\hfil
\includegraphics[page=1,width=0.95\textwidth]{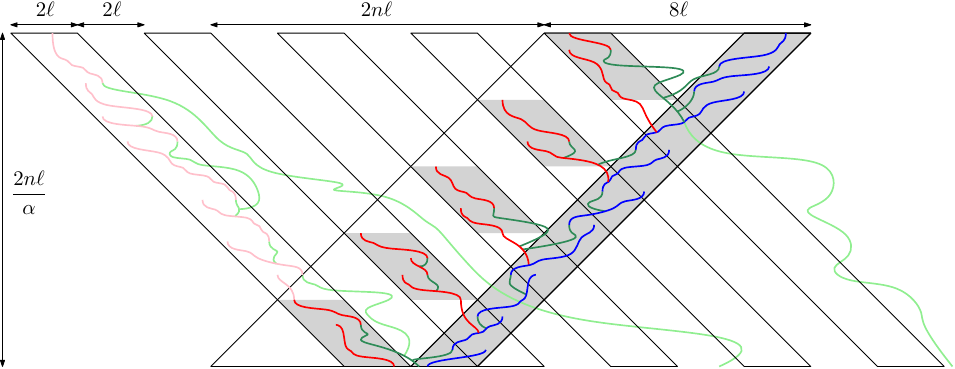}
\caption{%
Illustration for $ n=5 $ of how simultaneous occurrence of $ n+1 $ similar events results in a complete open branch in $ D_{4\ell,\frac{n}{2},\alpha} $.
Green paths are $ \lambda_i $-open, red and blue paths are $ \lambda_e$-open.
The existence of paths in the gray areas ensures that $ \lambda_i $-open edges in between can be used. The green paths may be needed as they connect and lead to blue and red paths.
\\
\hspace*{\fill}
(high resolution, color online)
}
\label{fig:block2}
\end{figure}

\begin{proof}
Given $ \varepsilon>0 $ and $T>0$,
let
$ \sE'_{\varepsilon,T} $ denote the event that
\[
|\cR\eta^-_t - \alpha t | \leq \varepsilon T\, \text{ for all } t\in[0,T].
\]
By Lemma~\ref{lemma:speed}, almost surely
$ |\cR\eta^-_t - \alpha t| \leq \varepsilon t $
for large enough $ t $.
Hence,
$ \Pb(\sE'_{\varepsilon,T}) \to 1 $ as $ T \to \infty $ for every $\varepsilon>0$.

Note the following about $ \sE'_{\varepsilon,T} $.
If this event occurs, then, for each $ t \in [0,T] $, there exists a $ \lambda_i $-open path $ \gamma $ from $ \Z^{-} \times \{0\} $ to $ (\cR\eta_t^-,t) $ such that $ \gamma_s \leq \cR \eta_s^- $ for all $ s \in [0,T] $ and each jump of $ \gamma $ to the right that coincides with a jump of $ \cR \eta_s^- $ to the right happens at a $ \lambda_e $-open edge.

For convenience,
let 
$ \sE'_{\varepsilon,T,x} $
denote the event that $ |\cR\eta^{(-\infty,x]}_t - \alpha t - x | \leq \varepsilon T\, \text{ for all } t\in[0,T] $.
In words, $ \sE'_{\varepsilon,T,x} $ is the same event as $ \sE'_{\varepsilon,T} $ except that the condition is translated by $ x $ units in space to the right.
Also, let
$ \sE''_{\varepsilon,T} $
denote the event that
$ |\cL \eta^+_t + \alpha t | \leq \varepsilon T $ for all $ t \in [0,T] $.
In words, $ \sE''_{\varepsilon,T} $ is the same event as $ \sE'_{\varepsilon,T} $ except that the condition is mirrored in space.
Finally, let $ \sE''_{\varepsilon,T,x} $ denote the translation of $ \sE''_{\varepsilon,T} $ by $ x $ units in space to the right.

Given positive integers $n$ and $\ell$, take $ \varepsilon = \frac{\alpha}{2n} $ and $T = \frac{2n\ell}{\alpha}$.
Note that the event $ \sE'_{\varepsilon,T} $ implies that the random points $ (\cR \eta^{-}_t, t) $ are in $ D_{\ell,2n,\alpha} $ for all $ t\in[0,T] $.
Combining several similar events together, we have
\[
\sE'_{\varepsilon,T,3\ell} \cap
\sE''_{\varepsilon,T,\ell} \cap
\sE''_{\varepsilon,T,5\ell} \cap
\sE''_{\varepsilon,T,9\ell} \cap \dots \cap
\sE''_{\varepsilon,T,(4n-3)\ell}
\subseteq
\sE_{4\ell,\frac{n}{2},\alpha,\lambda_i,\lambda_e}
\]
That is, simultaneous occurrence of $ n+1 $ events of the form $ \sE' $ and $ \sE'' $ imply existence of a complete open branch in $ D_{4\ell,\frac{n}{2},\alpha} $, as shown in Figure~\ref{fig:block2}.
In particular,
$ \Pb( \sE_{4\ell,\frac{n}{2},\alpha,\lambda_i,\lambda_e} ) \to 1 $ as $\ell \to \infty$ for each $n$ fixed.

Turning back to the statement of the lemma, we can take $n=2k$ and let $ L=4\ell\to\infty $,
so
we
get
$ \Pb( \sE_{L,k,\alpha,\lambda_i,\lambda_e} ) \to 1 $ as $L \to \infty$ in $4\N$,
which proves the lemma.
\end{proof}

\begin{proof}
[\proofname\ of Theorem~\ref{thm:critnospeed}]
\begin{figure}[b!]
\subfloat[]{\includegraphics[page=1,width=0.5\textwidth]{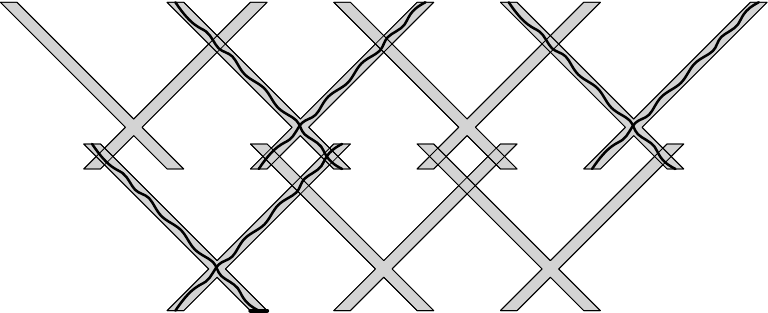}\label{fig:renormalizationa}}%
\subfloat[]{\includegraphics[page=2,width=0.5\textwidth]{renormalization}\label{fig:renormalizationb}}%
\caption{Good blocks and long paths.}
\label{fig:renormalizationz}
\end{figure}
We are assuming that $ \lambda_e \leq \lambda_i $ and $ \alpha=\alpha(\lambda_i,\lambda_e)>0 $.
Fix $ k=20 $ and, given $ L \in \N $, consider a grid made of copies of shear-shaped regions, as in Figure~\ref{fig:renormalizationa}.
Each shear consists of a translation of $ D_{L,k,\alpha} $ and a translation of $ D_{L,k,-\alpha} $.
(The choice of $ k=20 $ is not optimal, but it makes the picture more clear, especially considering that the figure was produced on a discrete space-time grid.)

We say that a shear is \emph{$ (\lambda_i,\lambda_e) $-good} if its blades contain complete $ (\lambda_i,\lambda_e) $-open branches.
Note that each shear only overlaps with its 6 nearest-neighboring shears (two below, two above, one on each side), and the goodness of shears in any collection not having nearest-neighbors is distributed as a product measure.
By~\cite{LiggettSchonmannStacey97}, there is $ \delta>0 $ such that, if $ \Pb(\sE_{L,k,\alpha,\lambda_i,\lambda_e})>1-\delta $, then the collection of $ (\lambda_i,\lambda_e) $-good shears stochastically dominates an i.i.d.\ configuration which is supercritical for oriented site percolation, illustrated in Figure~\ref{fig:renormalizationb}.
By Lemma~\ref{lemma:block}, we can fix $ L \in \N $ such that $ \Pb(\sE_{L,k,\alpha,\lambda_i,\lambda_e})>1-\delta $.
On the other hand, since $ D_{L,k,\alpha} $ is finite,
there exist $ \lambda_i'<\lambda_i $ and $ \lambda_e'<\lambda_e $ such that $ \lambda_e' \leq \lambda_i' $ and $ \Pb(\sE_{L,k,\alpha,\lambda_i',\lambda_e'})>1-\delta $ (note that we keep the same $ \alpha $).

So with positive probability there is an infinite oriented path of $ (\lambda_i',\lambda_e') $-good shears, such that the bottom of the first shear in this path is contained in $ A \times \{0\} $, where $ A := \{-6L,\dots,6L\} \subseteq \Z $.
We now argue that, when this event occurs, $ \eta^A_{0,t,\lambda_i',\lambda_e'}\ne\emptyset $ for all $ t\geq 0 $.
Indeed, suppose it occurs and let $ t>0 $.
Consider a shear whose time span includes $ t $ and which belongs to this infinite oriented path of $ (\lambda_i',\lambda_e') $-good shears.
By following this good path and applying Lemma~\ref{lemma:crossing} a finite number of times, we get a complete $ (\lambda_i',\lambda_e') $-open branch in a domain whose bottom is contained in $ A \times \{0\} $ and whose time span includes $ t $.
Hence, $ \eta^A_{0,t,\lambda_i',\lambda_e'} \ne\emptyset $ as claimed.

By Remark~\ref{rmk:allornothing}, $ \theta(\lambda_i',\lambda_e')>0 $.
Thus, $ \lambda_e > \lambda_*^e(\lambda_i) $
and
$ \lambda_i > \lambda_*^i(\lambda_e) $
whenever $ \lambda_e \leq \lambda_i $ and $ \alpha(\lambda_i,\lambda_e) >0 $.
Hence, for $ \lambda_e \leq \lambda_c $, since $ \lambda_*^i(\lambda_e) \geq \lambda_c \geq \lambda_e $, we have $\alpha(\lambda_*^i(\lambda_e),\lambda_e) \leq 0$.
Likewise, for $ \lambda_i \geq \lambda_c $, since $ \lambda_*^e(\lambda_i) \leq \lambda_c \leq \lambda_i $, we have $\alpha(\lambda_i,\lambda_*^e(\lambda_i)) \leq 0$.
On the other hand, by Lemmas~\ref{lemma:rightcontinuous} and~\ref{lemma:nonnegative} we have $\alpha(\lambda_*^i(\lambda_e),\lambda_e) \geq 0$ and $\alpha(\lambda_i,\lambda_*^e(\lambda_i)) \geq 0$, proving the theorem.
\end{proof}

\section{Speed increase with external infection rate}
\label{sec:speedincreasese}

In this section we prove Proposition~\ref{prop:speedincreasese}.
A small increase in the external infection rate can be seen as a forced infection of the site immediately to the right of the current configuration, with some small but positive frequency.
We want to show that this amounts to an increase in the growth speed.
To take advantage of this extra infected site, we compare the effect of infecting one site with the effect of shifting the configuration to the right (and the latter obviously amounts to a positive increase in the speed).
To achieve such a comparison (Proposition~\ref{prop:pathdomination}), we start from the invariant measure for the process seen from the rightmost infected site, provided by~Proposition~\ref{prop:existsinv}.


\emph{Notation.}
Throughout this section, letter $\eta$ means that parameters $(\lambda_i,\lambda_e)$ are being used,
letter $\xi$ means that parameters $(\lambda_i,\lambda_e+\varepsilon)$ are being used, via an extra Poisson clock of rate $\varepsilon$ to account for the increase in $\lambda_e$.
The values of $ \varepsilon>0 $ and $ (\lambda_i , \lambda_e) $ such that $\theta(\lambda_i,\lambda_e) > 0$ and $ \lambda_e + \varepsilon \leq \lambda_i $ are arbitrary but fixed.

\begin{proof}
[\proofname\ of Proposition~\ref{prop:speedincreasese}]
We will consider a sequence of coupled processes
 $(\eta^{n}_t,\xi^{n}_t)_{t \geq 0}$ indexed by $ n \in \N $.
Denote by $0<\tau_1<\tau_2<\dots$ the times of the external Poisson clock.

Start by defining the pair $ (\eta_t,\xi_t)_{t \geq 0} $ starting from $ \eta_0 = \xi_0 $ sampled from $ \mu $, which is supported on $ \Sigma^\odot $.
For all $t\in [0,\tau_1 )$ we have $\xi_t=\eta_t$ and at time $\tau_1$ we have
$$
\xi_{\tau_1}
=
\Xi \eta_{\tau_1}
=
\eta_{\tau_1} \cup \{\mathcal R \eta_{\tau_1}+1\}
,
$$
where $\Xi$ was defined in Section~\ref{sec:edge}.
Since the process $(\eta_t)_{t\geq 0}$ is independent of $\tau_1$, the distribution of $\Psi \eta_{\tau_1}$ is $\mu$.
Thus, by Proposition~\ref{prop:pathdomination}, $\Psi \xi_{\tau_1} \succcurlyeq \mu$.
Hence, enlarging the underlying probability space if necessary, there exists a random configuration $ \zeta^1 $
such that $ \Psi \zeta^1 $ is $\mu$-distributed, $ \cR \zeta^1 = \cR \xi_{\tau_1} $ and $ \zeta^1 \subseteq \xi_{\tau_1}$.
Moreover, $ \zeta^1 $ is independent of $ \tau_1 $ and of the graphical construction $ \omega $ after time $ \tau_1 $.

We now consider a coupled process $(\eta^1_{t},\xi^1_{t})_{t \geq \tau_1}$ which evolves using the same graphical construction and whose initial condition is $\eta^1_{\tau_1}=\xi^1_{\tau_1}=\zeta^1$.
By Lemma~\ref{lemma:attractive},
$
\eta^1_{\tau_1+t} \subseteq \xi^1_{\tau_1+t} \subseteq \xi_{\tau_1+t}
$
for all $ t \geq 0 $.
Moreover, since the process $(\eta^1_{\tau_1 +t} )_{t\geq 0}$ is equal in distribution to the process $(T \eta_{\tau_1 +t})_{t\geq 0}$ we have
$
\E [ \cR\eta^1_{\tau_1+t} ]
=
\E [ \cR\eta_{\tau_1+t} ]
+ 1.
$
Proceeding recursively, for each $ n $ we obtain processes $(\eta^n_{t},\xi^n_{t})_{t\geq \tau_n}$ such that
$
\eta^{n}_{\tau_n+t} \subseteq \xi^{n}_{\tau_n+t} \subseteq \xi_{\tau_n+t}
$
and
$
\E [ \cR\eta^n_{\tau_n+t} ]
=
\E [ \cR\eta_{\tau_n+t} ]
+ n
$
for all $ t \geq 0 $.
In particular, for all $ n \in \N $,
\[
\E[\cR \xi_{\tau_n}] \geq \E[\cR \eta_{\tau_n}] + n
.
\]

We now show from this inequality that
\[
\alpha(\lambda_i,\lambda_e+\varepsilon) \geq \alpha(\lambda_i,\lambda_e) + \varepsilon
.
\]

Since $ \E[\cR \eta_{t}] = \alpha(\lambda_i,\lambda_e) \cdot t $ by Proposition~\ref{prop:speedconstant}, and $ (\eta_t)_t $ is independent of $ (\tau_k)_k $, we have
\[
\E[\cR \eta_{\tau_n}] = \frac{\alpha(\lambda_i,\lambda_e)}{\varepsilon} \, n
,
\]
as $\tau_n$ is the $n$-th arrival of Poisson clock of rate $\varepsilon$.

This and the previous inequality imply that, for all $ n \in \N $,
\begin{equation*}\label{f}
\E\Big[\frac{\cR \xi_{\tau_n}}{n}\Big]
\geq
\frac{\alpha(\lambda_i,\lambda_e) + \varepsilon}{\varepsilon}
.
\end{equation*}
On the other hand, by Lemma~\ref{lemma:gspeed},
$ \frac{\cR \xi_{\tau_n}}{\tau_n} \to \alpha(\lambda_i,\lambda_e+\varepsilon) $
a.s.,
and therefore
\[
\frac{\cR \xi_{\tau_n}}{n}
\to
\frac{\alpha(\lambda_i,\lambda_e+\varepsilon)}{\varepsilon}
\quad
\text{ a.s.}
\]
To finish the proof, we want to combine the above equations, and for this we need to relate the a.s.\ limit of $ \frac{\cR \xi_{\tau_n}}{n} $ to its expectation.

Let $(Z_t)_{t\geq 0}$ be the process on $\Z$ that starts at the origin and jumps one unit to the right whenever the process $(\cR \xi_t)_{t\geq 0}$
jumps to the right.
Then $\cR \xi_t\leq Z_t$ a.s.\ for all $t\geq 0$.
Also, the sequence $\frac{Z_{\tau_n}}{n}$ converges a.s.\ and in mean to a constant $v$.
Applying Fatou's lemma to $\frac {Z_{\tau_n}-\cR \xi_{\tau_n}}{n}$, we get
\[
v-
\E \Big[\limsup_n \frac{\cR \xi_{\tau_n}}{n}\Big]
\leq
v-
\limsup_n \E \Big[\frac{\cR \xi_{\tau_n}}{n}\Big]
\]
Hence,
\[
\frac{\alpha(\lambda_i,\lambda_e) + \varepsilon}{\varepsilon}
\leq
\limsup_n \E \Big[\frac{\cR \xi_{\tau_n}}{n}\Big]
\leq
\E \Big[\limsup_n \frac{\cR \xi_{\tau_n}}{n}\Big]
=
\frac{\alpha(\lambda_i,\lambda_e+\varepsilon)}{\varepsilon}
,
\]
proving the proposition.
\end{proof}

\section{Speed increase with internal infection rates}
\label{sec:intinc}

In this section we prove Proposition~\ref{prop:speedincint}.
We want to show that a small increase in the internal infection rate amounts to an increase in the speed.
The argument here is rather indirect if compared to the previous section.
It considers an auxiliary process that has infection rate $ \lambda_e $ only at the rightmost site but still $ \lambda_i $ at the leftmost one, shows that this process survives with probability $ \theta_*>0 $, and then shows that adding an infected site to a given configuration at a certain time increases the expected position of the rightmost infected site by at least $ \theta_* $, at any future time.
This effect is cumulative, and from there we conclude that a process with larger $ \lambda_i $ has larger speed, on average.
Since speed can be defined as expected speed (Lemma~\ref{lemma:speed}), this is enough to derive Theorem~\ref{thm:nohorizontals}.


\emph{Notation.}
Throughout this section, letter $\eta$ means that parameters $(\lambda_i,\lambda_e)$ are being used, and $\xi$ means that parameters $(\lambda_i',\lambda_e)$ are being used.
An auxiliary process, denoted with $ \zeta $, has infections to the right of the rightmost infected site occurring at rate $ \lambda_e $ and other infections occurring at rate $ \lambda_i $.
The values of $ \lambda_i > \lambda_c $ and $ \lambda_i' > \lambda_i $ are fixed, and $ \lambda_e = \lambda_*^e(\lambda_i) \leq \lambda_c $.

Let
\[
\theta_* := \Pb(\zeta^0_t \neq \emptyset, \forall t\geq 0).
\]
\begin{lemma}
\label{lemma:survivalint}
We have
$ \theta_* > 0. $
\end{lemma}
\begin{proof}
By Lemma~\ref{lemma:speed} and Theorem~\ref{thm:critnospeed}, $ \frac{\cR \eta^-_t}{t} \to 0 $ a.s.
Hence, for any $\delta >0$
\[
\lim_n \Pb( \cR \eta^-_t \geq -n -\delta t \text{ for all } t \geq 0)=1.
\]
Similarly, $ \frac{\cL \zeta_t^ +}{t} \to -\alpha(\lambda_i,\lambda_i) $, so, for every $ \delta' < \alpha(\lambda_i,\lambda_i) $,
\[
\lim_n \Pb( \cL \zeta^+_t \leq n -\delta' t \text{ for all } t \geq 0)=1
.
\]
Now take some $ \delta $ and $ \delta' $ such that $ 0 < \delta < \delta' < \alpha(\lambda_i,\lambda_i) $, and take $ n $ such that the above probabilities are larger than $ \frac{1}{2} $.
Translating the latter event horizontally by $ -3n $, we get the event $ \{\cL \zeta^{\{-3n,\dots,0\}\cup \Z_+} \leq -2n - \delta' t \text{ for all } t \geq 0\} $.
Simultaneous occurrence of this event and $ \{\cR \eta^-_t \geq -n -\delta t \text{ for all } t \geq 0\} $ implies $ \{\zeta^{\{-3n,\dots,0\}}_t \ne \emptyset \text{ for all } t \geq 0 \} $.
By Remark~\ref{rmk:allornothing}, $ \theta_*>0 $.
\end{proof}

\begin{lemma}\label{expinc}
Let $ A \in \Sigma^\ominus $ and let $A'= A\cup \{x\}$ for some $ x > \cR A $.
Then,
\[
\E[ \cR\eta^{A'}_t ]
\geq
\E[ \cR\eta^A_t ] + \theta_*
\]
for all $ t \geq 0 $.
\end{lemma}

\begin{proof}
We will construct a coupling of two contact processes $\eta^A$ and $\eta^{A'}$ starting from $A$ and $A'$, respectively.
Let $y=\cR A$.

For the process $\eta^{A'}$, we first explore the active paths for configuration $A'$ starting at $(x,0)$.
We then translate this collection of active paths by $(y-x,0)$ and use them for the process $\eta^{A}$.
Call $\tau \in (0,+\infty]$ the global extinction time of these active paths.

Then construct the paths starting to the left of $x$ with the same graphical construction up to the time they reach an already constructed active path, i.e.,\ an active path starting at $(x,0)$ for $\eta^{A'}$ and at $(y,0)$ for $\eta^{A}$.

We now observe that, for $t<\tau$, we have $\cR \eta^{A'}_t=\cR \eta^{A}_t +(x-y)$ and, for $t \geq \tau$ we have $\eta^{A'}_t \supseteq \eta^{A}_t$ because
for any active path for $A'$ which did not meet the descendants of $x$
there is a corresponding active path for $A$ which did not meet the descendants of $y$.
Therefore,
\[
\cR\eta^{A'}_t - \cR\eta^{A}_t
\geq
\1_{\{\tau > t\}}
.
\]
This proves the desired inequality and thus the lemma.
\end{proof}

\begin{proof}
[\proofname\ of Proposition~\ref{prop:speedincint}]
First note that there exists $ \delta>0 $ such that, for all $ A \in \Sigma^\ominus $, at $ t=1 $ we have $\Pb(\cR \xi^A_1\geq \cR \eta^A_1+1)\geq \delta$ (proof omitted).

Now for each $ n \in \N $, consider the process $ (\xi^n_t)_{t \geq n} $ starting from $ \xi^n_n = \eta^-_n $.
Using Lemma~\ref{expinc}, one gets $ \E [ \cR \xi^-_t ] \geq \E [ \cR \xi^1_t ] + \delta\theta_* $ for all $t\geq 1$.

Likewise, $\Pb(\cR \xi^1_2 \geq \cR \eta_2^- + 1)\geq \delta$, hence for $ t \geq 2 $
\begin{equation}
\nonumber
\E[ \cR \xi^-_t ]
\geq
\E[ \cR \xi^1_t ] + \delta\theta_*
\geq
\E[ \cR \xi^2_t ] + 2\delta\theta_*
\geq
\E[ \cR \eta^-_t ] + 2\delta\theta_*
.
\end{equation}
The induction is clear and gives
\begin{equation}
\nonumber
\E[ \cR \xi^-_t ]
\geq
\E[ \cR \eta^-_t ] + n\delta\theta_*
\end{equation}
for all $ n\in\N $ and $ t \geq n $.
Therefore,
\begin{equation}
\nonumber
\alpha(\lambda_i',\lambda_*^e(\lambda_i))
\geq
\alpha(\lambda_i,\lambda_*^e(\lambda_i))
+ \delta\theta_*
= \delta\theta_*
>0
\end{equation}
by
Theorem~\ref{thm:critnospeed}
and
Lemma~\ref{lemma:survivalint},
concluding the proof.
\end{proof}

\section{Domination in the non-attractive case}
\label{sec:nonattractive}

In this section we prove Theorem~\ref{thm:helpout}.

\emph{Notation.}
Throughout this section,
$ \lambda_i = \lambda_e = \lambda_c $.
Letter $\xi$ means that parameters $(\lambda_i,\lambda_e+\varepsilon)$ are being used, via an external Poisson clock to account for the $ \varepsilon $ increase.
The value of $ \varepsilon>0 $ is arbitrary but fixed.
As usual, letter $\eta$ means that parameters $(\lambda_i,\lambda_e)$ are being used.

Denote by $0=\tau_0<\tau_1<\tau_2<\dots$ the times of the external Poisson clock (which has intensity $ \varepsilon $).
We will construct a sequence of processes denoted
$(\xi^i_t)_{t \geq \tau_i}$, for $i \in \N_0$.
Their initial distribution will be the measure $\mu$ provided by Proposition~\ref{prop:existsinv} for $\lambda_i=\lambda_e=\lambda_c$, and the superscripts will be used to enumerate them.
The process $(\xi^0_t)_{t\geq 0}$ is given by $(\xi_t)_{t\geq 0}$ which starts from a random configuration $ \xi^0_0 \sim \mu $.

Suppose the process
$ (\xi^{i-1})_{t \geq \tau_{i-1}} $
has been constructed.
We will show how to construct
$ (\xi^{i})_{t \geq \tau_{i}} $.

We start by describing how we sample $ \xi^i_{\tau_i} $, the initial configuration at time $\tau_i$ for
$(\xi^i_t)_{t \geq \tau_i}$,
from
the configuration
$ \xi^{i-1}_{\tau_i} $.
This configuration $ \xi^i_{\tau_i} $ will satisfy $ \xi^i_{\tau_i} \subseteq \xi^{i-1}_{\tau_i} $, $\cR\xi^i_{\tau_i} = \cR\xi^{i-1}_{\tau_i} $ and $ \Psi\xi^i_{\tau_i} \sim \mu $.
We will describe it rather explicitly, because later on we will need to argue that a certain sequence is stationary, and this will be key to prove Theorem~\ref{thm:helpout}.
First note that, by Proposition~\ref{prop:existsinv}, $ \Psi \xi^{i-1}_{\tau_i-} \sim \mu$.
Also note that $ \xi^{i-1}_{\tau_i} = \Xi (\xi^{i-1}_{\tau_i-}) $.
By Proposition~\ref{prop:pathdomination}, $ \Psi \xi^{i-1}_{\tau_i} \succcurlyeq \mu $.
We will show how to use $ \xi^{i-1}_{\tau_i} $ and some extra randomness to sample the configuration $ \xi^i_{\tau_i} $ having the above properties.
Without loss of generality, we assume that $\cR \xi^{i-1}_{\tau_i}=0$, otherwise just use $\Psi \xi^{i-1}_{\tau_i}$ and translate the resulting $\xi^i_{\tau_i}$ accordingly.
Let $(V^i_k)_{k\in\N}$ be an infinite sequence of independent random variables uniformly distributed on $[0,1]$ which are also independent of all the Poisson processes used previously.
Also let $\nu$ be a distribution on $ \Sigma^\odot \times \Sigma^\odot$ such that its first marginal is $\mu$, its second marginal is the distribution of $\xi^{i-1}_{\tau_i}$ and $\nu(\{(\eta,\xi): \eta \subseteq \xi\})=1$.
This measure $\nu$ has a conditional distribution of $\eta$ given $\xi$, given by a regular conditional distribution.
We start by declaring that $0 \in \xi^i_{\tau_i}$.
Once $\xi^i_{\tau_i} \cap \{-k,\dots,0\}$ has been determined, we declare that $-k-1 \in \xi^i_{\tau_i}$ if
$V^i_{k+1} \leq \nu( -k-1 \in \eta \,|\, \xi=\xi^{i-1}_{\tau_i},\, \eta=\xi^i_{\tau_i} \text{ on } \{-k,\dots,0\} ) $,
and that $-k-1 \not\in \xi^i_{\tau_i}$ otherwise.
This determines $\xi^i_{\tau_i}$ with the claimed properties.

To describe the evolution of $(\xi^i_t)_{t\geq \tau_i}$, we start by looking at the active paths starting from the space-time point $(\cR \xi^{{i-1}}_{\tau_i},\tau_i)$ for the configuration $ \xi^{{i-1}}_{\tau_i} $.
These are the paths whose jumps are supported by a $ \lambda_c $-open edge, plus an extra jump to the right at the rightmost active path at times $\tau_{i+1},\tau_{i+2},\dots$.
For simplicity we assume this set of paths is finite in time (the opposite already implies Theorem~\ref{thm:helpout} anyway).
Let $A_{i,0}$ be the set of points in these paths, let $\tau_{i,0}=\tau_i$ and let
\[
\tau_{i,1}=\sup \{s: A_{i,0}\cap \{(x,s), x\in \Z\}\neq \emptyset \}
.
\]
Then proceed by induction on $j$: first let $A_{i,j-1}$ be the set of points in the open paths starting at $(\cR \xi^{i-1}_{\tau_{i,j-1}},\tau_{i,j-1})$
and after that let \newline $\tau_{i,j}=\sup \{s: A_{i,j-1}\cap \{(x,s), x\in \Z\}\neq \emptyset \}$.

The construction of the process $\xi^i$ is now done as follows:
\begin{enumerate}
\item
In the time interval $[\tau_{i,j},\tau_{i,j+1})$, the open paths starting at $(\cR \xi^i_{\tau_{i,j}},\tau_{i,j})$ are given by $T^{\cR \xi^i_{\tau_{i,j}} - \cR \xi^{i-1}_{\tau_{i,j}}} A_{i,j}$.
That is, we translate the open paths starting
from the rightmost point of $\xi^{i-1}_{\tau_{i,j}}$ to the rightmost point of $\xi^i_{\tau_{i,j}}$.
\item
The open paths in the same time interval starting at the other points of $\xi^i_{\tau_{i,j}}$ are given by the Poisson processes of parameters $\lambda_i$ and $1$ until they merge with the open paths starting from the rightmost point
of $\xi^i_{\tau_{i,j}}$ (if this occurs). 
\end{enumerate}

The point of this perhaps intricate construction is that it simultaneously satisfies two properties.
First, the sequence $ (\cR \xi^n_{\tau_n} - \cR \xi^{n-1}_{\tau_{n-1}})_{n\in\N} $ is stationary.
Second, stated below, it allows a comparison between $\xi^{i+1}$ and $\xi^i$.

\begin{lemma}
\label{forbidden zones}
With the above construction of the processes $\xi^i$, we have
$\cR \xi^{i+1}_t \leq \cR \xi^i_t $ for all $ i \in \N_0 $ and all $t\geq \tau_{i+1}$.
\end{lemma}
\begin{proof}
We prove by induction that, for all $j\in \N\cup\{0\}$,
\[
\cR(\xi^{i+1}_t)\leq \cR (\xi^i_t)
\]
for all $t\in [\tau_{i+1,j},\tau_{i+1,j+1})$ and
$\xi^{i+1}_{\tau_{i+1,j}} \subseteq \xi^i_{\tau_{i+1,j}}$.
For $j=0$ we obviously have
 $\xi^{i+1}_{\tau_{i+1,0}} = \xi^{i+1}_{\tau_{i+1}} \subseteq \xi^i_{\tau_{i+1}}= \xi^i_{\tau_{i+1,0}}$ and
 $\cR (\xi^{i+1}_t)= \cR (\xi^i_t)$ for all $t\in [\tau_{i+1,0},\tau_{i+1,1})$. For the inductive step assume that
 $\cR (\xi^{i+1}_t) \leq \cR (\xi^i_t)$ 
 for all 
 $t\in [\tau_{i+1,k},\tau_{i+1,k+1})$ and that
 $\xi^{i+1}_{\tau_{i+1,k}} \subseteq \xi^i_{\tau_{i+1,k}}$. At time $\tau_{i+1,k+1}$ the open paths of $\xi^i $ starting at $(\cR (\xi^i_{\tau_{i+1,k}}), \tau_{i+1,k})$ and the open paths of $\xi^{i+1}$ starting at $(\cR (\xi^{i+1}_{\tau_{i+1,k}}), \tau_{i+1,k})$
 have vanished. Hence, the open paths starting at the left of these points which have merged with the paths starting at these points have also vanished. Since the open paths of $\xi^{i+1}$ starting at $(\cR (\xi^{i+1}_{\tau_{i+1,k}}), \tau_{i+1,k})$ are to the left of 
 the open paths of $\xi^i $ starting at $(\cR (\xi^i_{\tau_{i+1,k}}), \tau_{i+1,k})$ we have $\xi^{i+1}_{\tau_{i+1,k+1}}\subseteq \xi^i_{\tau_{i+1,k+1}}$ and this implies that 
 $\cR (\xi^{i+1}_t)\leq \cR (\xi^i_t)$ 
 for all $t\in [\tau_{i+1,k+1},\tau_{i+1,k+2})$ completing the inductive step and the proof of the lemma.
\end{proof}

To simplify the notation, let
$Y_n=\cR \xi^n_{\tau_n} - \cR \xi^{n-1}_{\tau_{n-1}}$.

\begin{lemma}
\label{lemma:sta-int}
The random variables $Y_n$ are integrable and $\E Y_n =1$.
\end{lemma}
\begin{proof}
Write $Y_n=X_n+1$, so $X_n$ is distributed as the increment of the position of the rightmost point of $(\eta^{\mu})_{t\geq 0}$ in the time interval $[0,\tau]$ where $\tau$ is an exponential random variable that is independent of the process $(\eta^{\mu})_{t\geq 0}$.
That $ \E X_n = 0 $ follows from Proposition~\ref{prop:critinvspeedzero}.
\end{proof}

\begin{proof}
[\proofname\ of Theorem~\ref{thm:helpout}]
We will prove that
\begin{equation}
\label{eq:liminf}
\Pb \big(\liminf_t \tfrac{\cR \xi^{\mu}_t }{t}>0 \big)>0
,
\end{equation}
and then discuss how to get Theorem~\ref{thm:helpout} therefrom.
Remark~\ref{rmk:abitmore} requires extra work and is handled in the next section.

From the definition of $Y_k$, we have
\[
\cR \xi^n_{\tau_{n}}
=
\sum_{k=1}^{n} Y_k
\]
Recall that $ (Y_n)_n $ is stationary.
By the Ergodic Theorem, $\frac {\cR (\xi^n_{\tau_{n}})}{n} \to Y$ a.s., for some (possibly random) $Y$.
Moreover, by Lemma~\ref{lemma:sta-int}, $\E Y=1$ .
Hence,
\[
\Pb \big(\lim_n \tfrac {\cR (\xi^n_{\tau_{n}})}{n}>0 \big) > 0
.
\]
On the other hand, by Lemma~\ref{forbidden zones}, we have $\cR \xi^\mu_{\tau_{n}} \geq \cR \xi^n_{\tau_{n}}$,
whence
\[
\Pb \big(\liminf_n \tfrac {\cR (\xi^\mu_{\tau_{n}})}{n}>0 \big) > 0
.
\]
Since $\sup_{t\in [\tau_n,\tau_{n+1}]}(\cR (\xi^\mu_{\tau_{n+1}})-\cR (\xi^\mu_t))$ is bounded above by a geometric random variable with parameter $\frac {\varepsilon}{\lambda_c+\varepsilon}$, using a simple Borel-Cantelli argument and $\frac{\tau_n}{n} \to \varepsilon^{-1}$ a.s.,\ we can conclude~\eqref{eq:liminf}.

Now, consider a process $(\zeta_t)_t$ that has rate $\lambda_e+\varepsilon$ for infections to the right of the rightmost infected site and $\lambda_i$ everywhere else.
We claim that $\Pb(\zeta^0_t\ne \emptyset \, \forall t\geq 0) > 0$.
Indeed, if this probability were zero then there would exist a diverging sequence of random times $t_n$ such that $\xi^{\mu}_{t_n} = \eta^{\mu}_{t_n}$ where
$\eta^{\mu}$ is the process with parameters $(\lambda_c,\lambda_c)$, and, since $\lim \frac{\cR \eta^\mu_t}{t}=0$ a.s.\ by Proposition~\ref{prop:critinvspeedzero}, this would imply that
$\liminf_t \frac{\cR \xi^\mu_t}{t}=0$ a.s.

To get Theorem~\ref{thm:helpout} from the above claim, note that, although a direct comparison between $(\zeta^0_t)_t$ and $(\eta^0_t)_t$ using the same graphical construction may break down due to lack of attractiveness, we can compare the probability that these processes survive using Lemma~2.2 of~\cite{DurrettSchinazi00}, thus proving Theorem~\ref{thm:helpout}.
\end{proof}

\section{Survival of the critical contact process on $\N$ with enhancement at the boundary}
\label{sec:survn}

Continuing from the previous section, we now bootstrap from the intermediate results used in the proof of Theorem~\ref{thm:helpout} so as to finally justify Remark~\ref{rmk:abitmore}.

Note that survival of the process $(\zeta_t)_t$ on $\Z$ is a weaker version of Remark~\ref{rmk:abitmore}, differing only on the underlying physical space being $\Z$ instead of $\Z_+$.
So it remains to show that, with positive probability, the process survives even if the negative half-line is shut down.

Let us make this approach more precise.
On the event $\{\zeta^0_t \ne \emptyset \ \forall t\geq 0\}$, define the random path $(\Gamma(t))_{t \geq 0}$ as the rightmost infinite path started from $(0,0)$.
That is, take
\[
\Gamma(t)=\max\{x \in \Z: (0,0)
\text{ infects }
(x,t) \text{ and } (x,t)
\text{ infects }
\infty \}
,
\]
which corresponds to an infinite infection path in the graphical construction, and is such that any other infinite infection path $\gamma$ starting from $\gamma(0)=0$ satisfies $\gamma(t)\leq \Gamma(t)$ for all $t\geq 0$.
Then our goal is to show that $\Pb(\Gamma(t) \geq 0 \ \forall t)>0$, since on this event the process survives regardless of the infections taking place on the negative half-line, which can be completely suppressed.

\begin{figure}[b!]
\hfil
\includegraphics[page=1,width=.95\textwidth]{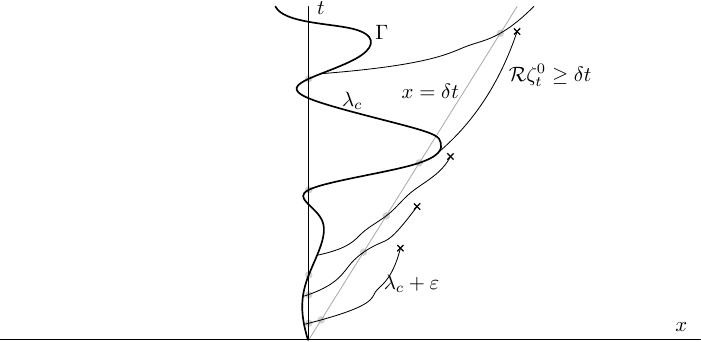}
\caption{%
By definition of $\Gamma$, every infection path to the right of $\Gamma$ is finite. Hence, in order to keep pushing the boundary so that $\cR \zeta^0_t \geq \delta t$ for all large $t$, infinitely many disjoint paths must start at $\Gamma$ and cross the line $x = \delta t$.
Since $\Gamma_t \leq 0$ infinitely many times, connecting pieces of $\Gamma$ and these disjoint paths will lead to infinitely many points $(0,t)$ which are connected to the line $x = \delta t$ via an infection path.
Since $\cR \zeta^0_t \geq \delta t$, the infection rate to the left of the line $x = \delta t$ is always $\lambda_c$, so these infection paths must be $\lambda_c$-open.
}
\label{fig:imcrossings}
\end{figure}

In the previous section we actually showed that $\{\zeta^0_t \ne \emptyset \ \forall t\geq0\} \cap \{\liminf_t \tfrac{\cR \xi^{\mu}_t }{t}>0\}$ has positive probability.
Now notice that, on the event $\{\zeta^0_t \ne \emptyset \ \forall t\geq0\}$, we have $\cR \zeta^0_t = \cR \xi^\mu_t  \ \forall t\geq0$, so combining these events we get
\[
\Pb( \cR \zeta^0_t \geq \delta t \ \forall t \geq T_0 ) > 0
\]
for some $T_0<\infty$ and some $\delta>0$.

So it is enough to show that the event $\{\cR \zeta^0_t \geq \delta t \ \forall t \geq T_0\} \cap \{\Gamma(t) \leq 0 \text{ i.o.}\}$ has probability zero.
Note that, on this event, there exists an increasing diverging sequence $(T_n)_n$ such that there are $\lambda_c$-open paths from $(0,T_n)$ to the line $\{(\lceil\delta t\rceil,t):t \geq 0\}$, see Figure~\ref{fig:imcrossings}.

To finish the proof of Remark~\ref{rmk:abitmore}, it remains to show that the latter event (infinitely many crossings of a cone by $\lambda_c$-open paths) has probability zero.

\begin{lemma}
Consider the critical contact process on $\Z$ and let $\delta > 0$.
The probability that there are infinitely many $\lambda_c$-open paths from $\{(0,t):t\geq 0\}$ to $\{(\lceil\delta t\rceil,t):t \geq 0\}$ is zero.
\end{lemma}
\begin{proof}
From Lemma~2 of~\cite{MountfordSweet00}, we get $0<c_1<C_1<\infty $ such that
\begin{equation}
\label{eq:expdecay1}
\Pb( \sup_{s \in [0,t]} \cR\eta^-_t \geq \delta t ) \leq C_1 e^{-c_1 t}
\end{equation}
for all $t \geq 0$.
Indeed, take $\lambda>\lambda_c$ such that $\alpha(\lambda,\lambda)<\delta$ and apply~\cite[Lemma~2]{MountfordSweet00} directly.

From~\eqref{eq:expdecay1}, we get
\[
\Pb( \sup_{s \in [t,2t]} \tfrac{\cR\eta^-_s}{s} \geq \delta ) \leq C_2 e^{-c_2 t}
\]
for all $t>0$ for new constants $0 < c_2 < C_2 < \infty$ (all constants depend on $\delta$).

Fix $t>0$.
Taking a union bound, the previous estimate gives
\[
\Pb( \sup_{s \geq t} \tfrac{\cR\eta^-_s}{s} \geq \delta )
\leq
\sum_{k \geq 1}
\Pb( \sup_{s \in [kt,2kt]} \tfrac{\cR\eta^-_s}{s} \geq \delta )
\leq
\sum_{k \geq 1}
C_2 e^{-c_2 k t}
\leq
C_3 e^{-c_3 t}
.
\]
Combining the previous estimates, we get
\begin{equation}
\label{eq:expdecay2}
\Pb( \sup_{s \geq 0} \tfrac{\cR\eta^-_s}{s+t} \geq \delta )
\leq
C_4 e^{-c_4 t}
,
\end{equation}
which is valid for all $t>0$.
We only need another tedious estimate, and the proof will be finished.

Fix $T>0$.
Let $A_T$ denote the event that there is a path from $\{0\}\times[T,+\infty)$ to $\{(\lceil\delta t\rceil,t):t \geq 0\}$.
Denote the recovery marks on $\{0\}\times[T,+\infty)$ by
$T+\tau_1,T+\tau_1+\tau_2,T+\tau_1+\tau_2+\tau_3,\dots$, and notice that $(\tau_n)_{n\in\N}$ are i.i.d.\ standard exponential variables.
Observe that $A_T \subseteq \cup_{n \in \N_0} A_{T,n}$, where $A_{T,n}$ is the event that there is a $\lambda_c$-open path from $T+\sum_{k=1}^n \tau_k$ to some $(x,t)$ with $ x \geq \delta t$.
Let $\beta = \E[e^{-c_4 \tau_1}] < 1$.
Using~\eqref{eq:expdecay2}, 
we get
\[
\Pb(A_{T,n})
=
\E[\Pb(A_{T_n}|\tau_1,\dots,\tau_n)]
\leq
\E[C_4 \exp(-c_4(T+\tau_1+\dots+\tau_n))]
\leq
C_4 e^{-c_4 T} \beta^n
,
\]
and thus
\[
\Pb(A_T)
\leq
\sum_{n=0}^\infty
\Pb(A_{T,n})
\leq
\frac{C_4 e^{-c_4 T}}{1-\beta}.
\]
Since this vanishes as $T \to \infty$, the lemma is proved.
\end{proof}
Using the above lemma, we conclude that $\Pb(\Gamma(t) \geq 0 \ \forall t)>0$, proving Remark~\ref{rmk:abitmore}.

\section{No survival for the critical contact process on $\N$ with finite enhancements}
\label{sec:zhang}
In this section we prove Theorem~\ref{thm:fixenhancement}.

The proof is by contradiction.
Let $k\in\N$, $\lambda<\infty$ and $\delta>0$.
Let $\zeta^0_{t;k,\lambda,\delta}$ denote the contact process on $\Z_+$ started from $\zeta_0=\{0\}$, having infection rate $\lambda$ and recovery rate $\delta$ at sites $\{0,1,\dots,k\}$, and infection rate $\lambda_c$ and recovery rate $1$ at sites $\{k+1,k+2,\dots\}$.
Suppose by contradiction that
$\Pb(\zeta^0_{t;k,\lambda,\delta}\ne\emptyset \,\forall t\geq 0)>0$.
Fix some $0 < \delta' < \delta'' < \delta$ and $\lambda' = \lambda + (\delta-\delta'')/2 < \infty$.
By defining both processes on the same graphical construction, we have
$\zeta^0_{t;k,\lambda',\delta'} \supseteq \zeta^0_{t;k,\lambda,\delta}$,
and therefore 
$\Pb(\zeta^0_{t;k,\lambda',\delta'}\ne\emptyset \,\forall t\geq 0)>0$.

We will now produce a new graphical construction with parameters $\delta$ and $\lambda$ from the graphical construction with parameters $\delta'$ and $\lambda'$, using extra randomness to resample some of the marks.
On this new graphical construction, we will define a process $(\xi_t)_{t\geq 0}$, which is thus distributed as $(\zeta^0_{t;k,\lambda,\delta})_{t \geq 0}$.
At each site $x\in\{0,\dots,k\}$, each $\lambda'$-infection mark from $x$ to $x\pm 1$ is converted into a recovery mark at $x$ with probability $1-\lambda/\lambda'$, independently for each mark; moreover, an independent Poisson clock of rate $\delta''-\delta'$ is added as new recovery marks.
This new graphical construction has the same distribution as the graphical construction with parameters $k$, $\lambda$ and $\delta$, and we use it to define the process $(\xi_t)_t$.
Below we show that $\Pb(\xi_t\ne\emptyset \, \forall t\geq0)=0$, contradicting the assumption that $\Pb(\zeta^0_{t;k,\lambda,\delta}\ne\emptyset \,\forall t\geq 0)>0$.


On the event $\{\zeta^0_{t;k,\lambda',\delta'} \ne \emptyset \,\forall t\geq 0\}$, let $\Gamma$ be the rightmost infinite open path started from $(0,0)$.
More precisely, and explicitly, take
\[
\Gamma(t)=\max\{x \in \Z_+: (0,0)
\underset{(k,\lambda',\delta')}{\rightsquigarrow}
(x,t) \text{ and } (x,t)
\underset{(k,\lambda',\delta')}{\rightsquigarrow}
\infty \}
,
\]
which corresponds to an infinite open path in the graphical construction, and is such that any other infinite open path $\gamma$ starting from $\gamma(0)=0$ satisfies $\gamma(t)\leq \Gamma(t)$ for all $t\geq 0$.
In the above definition, we are using the notation $(x,t) {\rightsquigarrow}_{(k,\lambda',\delta')} (y,s)$ to denote existence of a path from $(x,t)$ to $(y,s)$ in the model that uses parameters $(k,\lambda',\delta')$.

Now observe that, almost surely, $\Gamma(t)\leq k$ for arbitrarily large times, since otherwise we would have an infinite open path in the region $\{k+1,k+2,\dots\} \times [0,+\infty)$, contradicting~\cite{BezuidenhoutGrimmett90}.
So there is a divergent sequence of times $(t_n)_n$ such that $\Gamma(t_n) \leq k$ for all $n$.
We can assume that $t_{n+1}>t_n+1$ for all $n$.

Finally, let $\mathcal{A}_n$ denote the event that, for each site $x\in\{0,\dots,k\}$, during the time interval $[t_n,t_{n+1}]$ a recovery mark is added and the first infection mark (if there is one) is converted into a recovery mark.
The events $(\mathcal{A}_n)_n$ are conditionally independent given the original graphical construction, and have (conditional) probability at least $((1-\lambda/\lambda')(1-e^{\delta'-\delta''}))^{k+1}>0$.
In particular, a.s.\ at least one such event will occur.
Now observe that the occurrence of $\mathcal{A}_n$ for some $n$ is enough to simultaneously break every infinite open path that starts from $(0,0)$.
We have thus proved that $\Pb(\xi_t\ne\emptyset \, \forall t\geq0)=0$ as announced above, contradicting the assumption that this probability is positive.
This proves Theorem~\ref{thm:fixenhancement}.

\section{Open problems}
\label{sec:openproblems}

\begin{enumerate}
\item
As already said in Remark~\ref{rmk:zhang2}, it is not known whether the contact process on $\Z$ with parameter $\lambda_c$ dies out if we increase the rate of infection emanating from a fixed site.
\item
The proofs of the asymptotic speed results given in Section~\ref{sec:speedetc} depend on the Subadditive Ergodic Theorem. Do they still hold when $\lambda_i < \lambda_e$? The difficulty is due to the loss of attractiveness of the process.
\item
Does the conclusion of Proposition~\ref{prop:existsinv} still hold when $\lambda_i=\lambda_c< \lambda_e$?
\end{enumerate}

\begin{acks}
This paper was written while the first author was visiting IMPA in Rio de Janeiro. Thanks are given for its hospitality.
L.R. has been supported by FAPESP grants 2023/13453-5 and 2025/27064-6 and CNPq grant 408590/2024-6.
\end{acks}


\begin{thebibliography}{16}

\bibitem{AizenmanGrimmett91}
\begin{barticle}[author]
\bauthor{\bsnm{Aizenman},~\bfnm{Michael}\binits{M.}} \AND
  \bauthor{\bsnm{Grimmett},~\bfnm{Geoffrey}\binits{G.}}
(\byear{1991}).
\btitle{Strict monotonicity for critical points in percolation and
  ferromagnetic models}.
\bjournal{J. Statist. Phys.}
\bvolume{63}
\bpages{817--835}.
\bdoi{10.1007/BF01029985}
\end{barticle}
\endbibitem

\bibitem{BezuidenhoutGray94}
\begin{barticle}[author]
\bauthor{\bsnm{Bezuidenhout},~\bfnm{Carol}\binits{C.}} \AND
  \bauthor{\bsnm{Gray},~\bfnm{Lawrence}\binits{L.}}
(\byear{1994}).
\btitle{Critical attractive spin systems}.
\bjournal{Ann. Probab.}
\bvolume{22}
\bpages{1160--1194}.
\bdoi{10.1214/aop/1176988599}
\end{barticle}
\endbibitem

\bibitem{BezuidenhoutGrimmett90}
\begin{barticle}[author]
\bauthor{\bsnm{Bezuidenhout},~\bfnm{C.}\binits{C.}} \AND
  \bauthor{\bsnm{Grimmett},~\bfnm{G.}\binits{G.}}
(\byear{1990}).
\btitle{The Critical Contact Process Dies Out}.
\bjournal{Ann. Probab.}
\bvolume{18}
\bpages{1462--1482}.
\bdoi{10.1214/aop/1176990627}
\end{barticle}
\endbibitem

\bibitem{CoxDurrettSchinazi91}
\begin{barticle}[author]
\bauthor{\bsnm{Cox},~\bfnm{J.~T.}\binits{J.~T.}},
  \bauthor{\bsnm{Durrett},~\bfnm{R.}\binits{R.}} \AND
  \bauthor{\bsnm{Schinazi},~\bfnm{R.}\binits{R.}}
(\byear{1991}).
\btitle{The critical contact process seen from the right edge}.
\bjournal{Probab. Theory Related Fields}
\bvolume{87}
\bpages{325--332}.
\bdoi{10.1007/BF01312213}
\end{barticle}
\endbibitem

\bibitem{DurrettSchinazi00}
\begin{barticle}[author]
\bauthor{\bsnm{Durrett},~\bfnm{Rick}\binits{R.}} \AND
  \bauthor{\bsnm{Schinazi},~\bfnm{Rinaldo~B.}\binits{R.~B.}}
(\byear{2000}).
\btitle{Boundary modified contact processes}.
\bjournal{J. Theoret. Probab.}
\bvolume{13}
\bpages{575--594}.
\bdoi{10.1023/A:1007881121529}
\end{barticle}
\endbibitem

\bibitem{GalvesPresutti87}
\begin{barticle}[author]
\bauthor{\bsnm{Galves},~\bfnm{Antonio}\binits{A.}} \AND
  \bauthor{\bsnm{Presutti},~\bfnm{Errico}\binits{E.}}
(\byear{1987}).
\btitle{Edge fluctuations for the one-dimensional supercritical contact
  process}.
\bjournal{Ann. Probab.}
\bvolume{15}
\bpages{1131--1145}.
\bdoi{10.1214/aop/1176992086}
\end{barticle}
\endbibitem

\bibitem{Griffeath81}
\begin{barticle}[author]
\bauthor{\bsnm{Griffeath},~\bfnm{David}\binits{D.}}
(\byear{1981}).
\btitle{The basic contact processes}.
\bjournal{Stochastic Process. Appl.}
\bvolume{11}
\bpages{151--185}.
\bdoi{10.1016/0304-4149(81)90002-8}
\end{barticle}
\endbibitem

\bibitem{Griffeath83}
\begin{barticle}[author]
\bauthor{\bsnm{Griffeath},~\bfnm{David}\binits{D.}}
(\byear{1983}).
\btitle{The binary contact path process}.
\bjournal{Ann. Probab.}
\bvolume{11}
\bpages{692--705}.
\bdoi{10.1214/aop/1176993514}
\end{barticle}
\endbibitem

\bibitem{Grimmett99}
\begin{bbook}[author]
\bauthor{\bsnm{Grimmett},~\bfnm{G.~R.}\binits{G.~R.}}
(\byear{1999}).
\btitle{Percolation},
\bedition{2} ed.
\bseries{Grundlehren der mathematischen Wissenschaften}
\bvolume{321}.
\bpublisher{Springer-Verlag}.
\bdoi{10.1007/978-3-662-03981-6}
\end{bbook}
\endbibitem

\bibitem{HolleyLiggett81}
\begin{barticle}[author]
\bauthor{\bsnm{Holley},~\bfnm{Richard}\binits{R.}} \AND
  \bauthor{\bsnm{Liggett},~\bfnm{Thomas~M.}\binits{T.~M.}}
(\byear{1981}).
\btitle{Generalized potlatch and smoothing processes}.
\bjournal{Z. Wahrsch. Verw. Gebiete}
\bvolume{55}
\bpages{165--195}.
\bdoi{10.1007/BF00535158}
\end{barticle}
\endbibitem

\bibitem{Liggett05}
\begin{bbook}[author]
\bauthor{\bsnm{Liggett},~\bfnm{Thomas~M.}\binits{T.~M.}}
(\byear{2005}).
\btitle{Interacting particle systems}.
\bseries{Classics in Mathematics}.
\bpublisher{Springer-Verlag, Berlin}.
\bdoi{10.1007/b138374}
\end{bbook}
\endbibitem

\bibitem{LiggettSchonmannStacey97}
\begin{barticle}[author]
\bauthor{\bsnm{Liggett},~\bfnm{T.~M.}\binits{T.~M.}},
  \bauthor{\bsnm{Schonmann},~\bfnm{R.~H.}\binits{R.~H.}} \AND
  \bauthor{\bsnm{Stacey},~\bfnm{A.~M.}\binits{A.~M.}}
(\byear{1997}).
\btitle{Domination by product measures}.
\bjournal{Ann. Probab.}
\bvolume{25}
\bpages{71--95}.
\bdoi{10.1214/aop/1024404279}
\end{barticle}
\endbibitem

\bibitem{MountfordSweet00}
\begin{barticle}[author]
\bauthor{\bsnm{Mountford},~\bfnm{Thomas~S.}\binits{T.~S.}} \AND
  \bauthor{\bsnm{Sweet},~\bfnm{Ted~D.}\binits{T.~D.}}
(\byear{2000}).
\btitle{An extension of {K}uczek's argument to nonnearest neighbor contact
  processes}.
\bjournal{J. Theoret. Probab.}
\bvolume{13}
\bpages{1061--1081}.
\bdoi{10.1023/A:1007818108889}
\end{barticle}
\endbibitem

\bibitem{Terra24}
\begin{bmisc}[author]
\bauthor{\bsnm{Terra},~\bfnm{C\'{e}lio}\binits{C.}}
(\byear{2024}).
\btitle{Dynamic Phenomena in Interacting Particle Systems: Phase Transitions
  and Equilibrium}.
\bnote{PhD thesis. arXiv:2412.16601}.
\end{bmisc}
\endbibitem

\bibitem{Valesin24}
\begin{barticle}[author]
\bauthor{\bsnm{Valesin},~\bfnm{D.}\binits{D.}}
(\byear{2024}).
\btitle{The Contact Process on Random Graphs}.
\bjournal{Ensaios Matem{\'a}ticos}
\bvolume{40}
\bpages{1--115}.
\bdoi{10.21711/217504322024/em401}
\end{barticle}
\endbibitem

\bibitem{Zhang94}
\begin{barticle}[author]
\bauthor{\bsnm{Zhang},~\bfnm{Yu}\binits{Y.}}
(\byear{1994}).
\btitle{A note on inhomogeneous percolation}.
\bjournal{Ann. Probab.}
\bvolume{22}
\bpages{803--819}.
\bdoi{10.1214/aop/1176988730}
\end{barticle}
\endbibitem

\end{thebibliography}

%
%
%
%

%
%

\end{document}